\documentclass[twoside,11pt,reqno]{amsart}
\usepackage{amsmath,amssymb,amscd,mathrsfs,epic,empheq}
\usepackage{latexsym,amsthm,amscd}

\makeatletter

\@addtoreset{equation}{section}

\newtheorem{Proposition}{Proposition}[section]
\newtheorem{Lemma}[Proposition]{Lemma}
\newtheorem{Theorem}[Proposition]{Theorem}
\newtheorem{Corollary}[Proposition]{Corollary}

\newtheorem{Remark}[Proposition]{Remark}

\def\phantomsubsection#1{\vspace{2mm}\noindent{\bf #1.}}

\newcommand{\bla}{\mbox{\boldmath$\lambda$}}
\newcommand{\bmu}{\mbox{\boldmath$\mu$}}

\newcommand{\bnu}{\mbox{\boldmath$\nu$}}

\newcommand{\cC}{C}
\newcommand{\cP}{P}
\newcommand{\cQ}{Q}

\newbox\squ  
\setbox\squ=\hbox{\vrule width.6pt
   \vbox{\hrule height.6pt width.4em\kern1ex
         \hrule height.6pt}
   \vrule width.6pt}

\def\deg{\operatorname{deg}}
\def\cups{\operatorname{cups}}
\def\caps{\operatorname{caps}}
\def\ups{\operatorname{\up}}
\def\downs{\operatorname{\down}}
\def\defect{\operatorname{def}}

\def\cl{{\operatorname{cl}}}
\def\ex{{\operatorname{ex}}}
\def\op{\operatorname{op}}

\def\down{\vee}
\def\up{\wedge}

\def\mod#1{\operatorname{Mod}_{l\!f}(#1)}

\def\C{{\mathbb F}}

\def\Z{{\mathbb Z}}

\def\0{{\bar 0}}
\def\1{{\bar 1}}

\def\hom{{\operatorname{Hom}}}
\def\ext{{\operatorname{Ext}}}
\def\End{{\operatorname{End}}}

\def\im{{\operatorname{im}}}
\def\soc{{\operatorname{soc}\:}}

\def\eps{{\varepsilon}}
\def\phi{{\varphi}}

\def\la{{\lambda}}
\def\La{{\Lambda}}
\def\Ga{{\Gamma}}
\def\ga{{\gamma}}

\def\De{{\Delta}}
\def\al{{\alpha}}
\def\be{{\beta}}

\def\bt{\text{\boldmath$t$}}
\def\br{\text{\boldmath$r$}}
\def\bs{\text{\boldmath$s$}}

\def\bu{\text{\boldmath$u$}}

\def\bLa{\text{\boldmath$\La$}}

\def\bGa{\text{\boldmath$\Ga$}}
\def\bPi{\text{\boldmath$\Pi$}}

\def\bDe{\text{\boldmath$\De$}}
\def\bOm{\text{\boldmath$\Omega$}}
\def\bUps{\text{\boldmath$\Upsilon$}}

\def\bGa{\text{\boldmath$\Ga$}}
\def\bla{\text{\boldmath$\la$}}
\def\bmu{\text{\boldmath$\mu$}}

\begin{document}

\title[Khovanov's diagram algebra II]{Highest weight categories
arising from Khovanov's diagram algebra II: Koszulity}
\author{Jonathan Brundan and Catharina Stroppel}

\address{Department of Mathematics, University of Oregon, Eugene, OR 97403, USA}
\email{brundan@uoregon.edu}
\address{Department of Mathematics, University of Bonn,
53115 Bonn, Germany}
\email{stroppel@math.uni-bonn.de}

\thanks{First author supported in part by NSF grant no. DMS-0654147}
\thanks{Second author supported by the NSF and the Minerva Research Foundation DMS-0635607.}

\begin{abstract}
This is the second of a series of four articles
studying various generalisations
of Khovanov's diagram algebra.
In this article we develop the general theory of
Khovanov's diagrammatically defined ``projective functors''
in our setting.
As an application, we give a direct proof of the fact that
the quasi-hereditary covers of generalised Khovanov algebras
are Koszul.
\end{abstract}
\maketitle

\tableofcontents

\section{Introduction}\label{sintroduction}

This is Part II of a series of four articles studying some generalisations
of Khovanov's diagram algebra.
In Part I, we explained how to associate to every block
$\La$ of weights
two positively graded associative algebras denoted
$H_\La$ and $K_\La$. Roughly speaking, a {\em weight} means a diagram like
$$
\begin{picture}(30,15)
\put(-95.5,1){$\la = $}
\put(-64.8,3){\line(1,0){184}}
\put(-68.2,1){$\scriptstyle\times$}
\put(-44.8,0.2){$\circ$}
\put(46.8,1){$\scriptstyle\times$}
\put(93.2,0.2){$\circ$}
\put(-21.7,-1.6){$\scriptstyle\up$}
\put(70.3,-1.6){$\scriptstyle\up$}
\put(1.3,3.1){$\scriptstyle\down$}
\put(116.3,3.1){$\scriptstyle\down$}
\put(24.3,3.1){$\scriptstyle\down$}
\end{picture}
$$
consisting of a (possibly infinite) number line with some vertices labelled $\down, \up,
\circ$ and $\times$. Then a {\em block} $\La$
is an equivalence class of weights,
two weights being equivalent if one is obtained from the other by
permuting $\down$'s and $\up$'s (but not changing the positions of
$\circ$'s and $\times$'s that only play an auxiliary role).
If $\La$ consists of weights with exactly
$n$ $\down$'s and $n$ $\up$'s, then
$H_\La$ is Khovanov's diagram algebra $H_n^n$ from
\cite[$\S$2.5]{K2} which plays a key role in the definition of
 Khovanov homology in knot theory.
We showed in Part I that $H_\La$ is a
symmetric algebra with a cellular basis
parametrised by various closed oriented circle
diagrams and that $K_\La$ is a quasi-hereditary algebra with a cellular basis
parametrised by arbitrary oriented circle diagrams.
We also described explicitly the combinatorics of
projective, irreducible and cell modules.

In \cite[$\S$2.7]{K2}, Khovanov introduced a class of
$(H_m^m,H_n^n)$-bimodules he called
{\em geometric bimodules}, one for each crossingless matching
between $2m$ points and $2n$ points. In this article
we consider the analogous classes of
bimodules for the algebras $H_\La$ and $K_\La$
in general. Tensoring with these bimodules gives rise to
some remarkable functors
on the module categories, which in the case of $K_\La$ we call
{\em projective functors}.
One of the main aims of this article is to
give a systematic account of
the general theory of such functors since, although elementary in nature,
proofs of many of
these foundational results are hard to find in the literature.
We then apply the theory to
give self-contained diagrammatic proofs of the
following:

\phantomsubsection{Koszulity}
We show that the algebras $K_\La$ are Koszul, establishing at the same time
that their Kazhdan-Lusztig polynomials
in the sense of Vogan \cite{V}
are the usual
Kazhdan-Lusztig polynomials associated to Grassmannians
which were computed explicitly by Lascoux and Sch\"utzenberger in \cite{LS}.

\phantomsubsection{Double centralizer property} We
prove a double centralizer property
which implies that
$K_\La$ is a quasi-hereditary cover of
$H_\La$
in the sense
of \cite[$\S$4.2]{Rou}.

\phantomsubsection{Kostant modules}
We classify and study the {Kostant modules} over the algebras $K_\La$
in the sense of \cite{BH}.
These are the irreducible modules whose Kazhdan-Lusztig polynomials
are multiplicity-free.
We show that they
possess a resolution
by direct sums of
cell modules in the spirit of
\cite{BGG}.

\vspace{2mm}
When $K_\La$ is finite dimensional,
these should not be regarded as new results, but the existing
proofs are indirect.
In fact, in the finite dimensional cases,
work of Braden \cite{Br} and the second author \cite[Theorem 5.8.1]{S}
shows that the category of
$K_\La$-modules is equivalent to the category of
perverse sheaves on a Grassmannian
(constructible with respect to the Schubert stratification).
In turn, by the Beilinson-Bernstein
localisation theorem from \cite{BB},
this category of perverse sheaves
is equivalent to a regular integral
block of a parabolic analogue of the Bernsein-Gelfand-Gelfand category
$\mathcal O$ associated to a maximal parabolic in type A.
Using this connection to parabolic category $\mathcal O$,
the fact that $K_\La$ is Koszul in the finite dimensional case
follows as a relatively trivial special case of
an important theorem of Beilinson, Ginzburg and Soergel
\cite[Theorem 1.1.3]{BGS} (see also \cite[Theorem 1.1]{Ba}).
Also the double centralizer property for $K_\La$
follows from \cite[Theorem 10.1]{Sdcp}, while
our results about Kostant modules in this setting (where
they correspond to the rationally smooth
Schubert varieties in Grassmannians)
follow from \cite{Lep, BH}.

The results of this article play an essential role in our subsequent work.
In Part III, we apply the
Koszulity and the
double centralizer property to give a direct algebraic proof
of the equivalence of the category of (graded) $K_\La$-modules
with the aforementioned blocks of (graded) parabolic category $\mathcal O$
in the finite dimensional cases;
this is one of the reasons we wanted to reprove
the above results independently of \cite{BB}, \cite{BGS} and \cite{Br} in the first place.
Then in Part IV, we relate some of the infinite
dimensional versions of the algebra $K_\La$ to blocks of the general linear
supergroup, a setting in which (so far) no geometric approach is available.
Since
we also establish Koszulity for the (locally unital) algebras $K_\La$
in these infinite dimensional cases, this implies
that blocks of the general linear supergroup are Koszul.
Finally our treatment of Kostant modules
leads to another proof of
the main result of \cite{CKL}: all
irreducible polynomial representations of $GL(m|n)$ possess BGG-type
resolutions.

\vspace{2mm}
\noindent
{\em Notation.}
Throughout the article we work over
a fixed ground field $\C$, and
gradings mean $\Z$-gradings.
For a graded algebra $K$, we write $\mod{K}$ for the
category of
locally finite dimensional graded {\em left} $K$-modules that are bounded
below; see \cite[$\S$5]{BS} for other general
conventions regarding graded modules.

\section{More diagrams}

We assume the reader is familiar with the
notions of weights, blocks, cup diagrams, cap diagrams, and circle
diagrams as defined in \cite[$\S2$]{BS}.
In particular, our cup and cap diagrams are allowed to
contain rays
as well as cups and caps, and they may also have some free vertices which are
not joined to anything.
In this section we generalise these
diagrammatic notions by incorporating
crossingless matchings.

\phantomsubsection{Crossingless matchings}
A {\em crossingless matching} means a diagram $t$ obtained by
drawing a cap diagram $c$
underneath a cup diagram $d$, and then connecting rays in $c$ to
rays in $d$ as prescribed by some order preserving
 bijection between the vertices of $c$ at the bottoms of rays
and the vertices of $d$ at the tops of rays. This is possible
if and only if $c$ and $d$ have the same number of rays;
if this common number of rays is finite there is a unique such
order preserving bijection.
Any crossingless matching is a union of cups, caps, and
{\em line segments} (which arise when two rays are connected).
For example:
$$
\begin{picture}(30,70)
\put(-42,66){\line(1,0){138.3}}
\put(72.5,66){\oval(47.2,46)[b]} \put(15.6,66){\oval(23,23)[b]}
\put(-18.8,66){\line(0,-1){14}} \put(-18.8,52){\line(4,-1){137.8}}
\put(119,1){\line(0,1){16.5}} \put(-41.8,45){\line(-1,-1){22.9}}
\put(-41.8,66){\line(0,-1){21.1}} \put(-64.8,1){\line(0,1){21.1}}
\put(-65,1){\line(1,0){184.2}}
\put(96,1){\line(0,1){2.5}}
\put(73,66){\line(0,-1){2.5}}
\put(15.5,1){\oval(115,53)[t]} \put(15.5,1){\oval(23,23)[t]}
\put(15.5,1){\oval(69,38)[t]}
\end{picture}
$$
Let $\cups(t)$ (resp.\ $\caps(t)$) denote the number of cups (resp.\ caps) in $t$. Also let $t^*$ denote the mirror image
of
$t$ in a horizontal axis.

Suppose that we are given blocks
$\La$ and $\Ga$ and a crossingless matching $t$.
We say that $t$
is a {\em $\La\Ga$-matching} if
the bottom and top
number lines of $t$ are the same as the number lines underlying
weights from $\La$ and $\Ga$, respectively.
More generally, suppose we are given a sequence
$\bLa = \La_k \cdots \La_0$ of blocks.
A {\em $\bLa$-matching} means a diagram of the
form $\bt = t_k \cdots t_1$ obtained by gluing
a sequence $t_1,\dots,t_k$
of crossingless matchings together  from top (starting with $t_1$)
to bottom (ending with $t_k$) such that
\begin{itemize}
\item
$t_i$ is a $\La_i \La_{i-1}$-matching for each $i=1,\dots,k$;
\item
the free vertices at the bottom of $t_i$
are in the same positions as the free vertices at the top
of $t_{i+1}$ for each $i=1,\dots,k-1$.
\end{itemize}
Given in addition cup and cap diagrams
$a$ and $b$ whose number lines are the same as the
bottom and top number lines of $t_k$ and $t_1$, respectively,
with free vertices in all the same positions,
we can glue
to obtain a {\em $\bLa$-circle diagram} $a \bt b = a t_k \cdots t_1 b$.
Here is an example with $k=2$:
$$
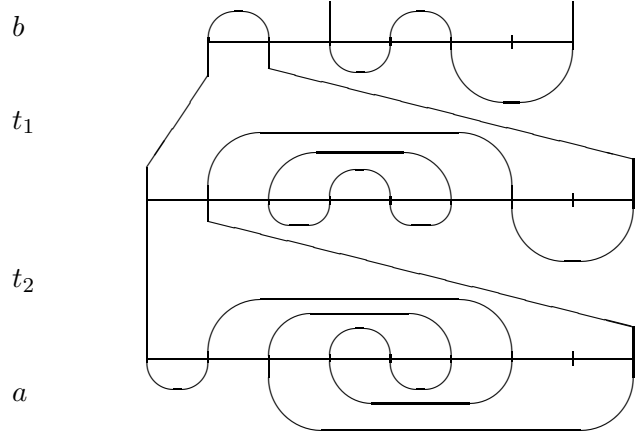
\begin{picture}(30,160)
\put(-116,10){$a$}
\put(-116,148){$b$}
\put(-116,113){$t_1$}
\put(-116,53){$t_2$}
\put(-30.4,146){\oval(23,23)[t]}
\put(4.1,146){\line(0,1){15}}
\put(96.1,146){\line(0,1){15}}
\put(38.6,146){\oval(23,23)[t]}
\put(15.6,146){\oval(23,23)[b]}
\put(-65.2,26){\line(1,0){184.2}}
\put(-65.2,86){\line(1,0){184.2}} \put(-42.2,146){\line(1,0){138.4}}
\put(15.5,86){\oval(115,51)[t]} \put(15.5,86){\oval(23,23)[t]}
\put(15.5,86){\oval(69,36)[t]} \put(119,86){\line(0,1){15.6}}
\put(-64.8,26){\line(1,0){184}} \put(-42,146){\line(0,-1){13}}
\put(-42,133){\line(-2,-3){23}} \put(-65,86){\line(0,1){12.5}}
\put(73,148.5){\line(0,-1){5}}
\put(73,146){\oval(46,46)[b]}
\put(-19,146){\line(0,-1){10}} \put(-18.9,136){\line(4,-1){137.8}}
\put(119,26){\line(0,1){12.1}} \put(-65,86){\line(0,-1){60}}
\put(96,86){\oval(46,46)[b]} \put(38.5,86){\oval(23,19)[b]}
\put(-7.5,86){\oval(23,19)[b]} \put(-42,86){\line(0,-1){7.9}}
\put(-41.9,78.2){\line(4,-1){160.7}}
\put(96,23.5){\line(0,1){5}}
\put(96,83.5){\line(0,1){5}}
\put(15.5,26){\oval(115,45)[t]}
\put(15.5,26){\oval(23,23)[t]}
\put(15.5,26){\oval(69,34)[t]}
\put(50,26){\oval(138,54)[b]}
\put(38.5,26){\oval(69,34)[b]}
\put(38.5,26){\oval(23,23)[b]}
\put(-53.5,26){\oval(23,23)[b]}
\end{picture}
$$
Any such circle diagram is a union of circles and lines.
We call it a {\em closed} circle diagram if there are
only circles, no lines.

\phantomsubsection{Orientations and degrees}
Let $\La$ and $\Ga$ be blocks and $t$ be a $\La\Ga$-matching.
Given weights $\la \in \La$ and $\mu \in \Ga$,
we can glue $\la, t$ and $\mu$ together from bottom to top to obtain
a new diagram $\la t \mu$.
This is called an {\em oriented $\La\Ga$-matching} if
\begin{itemize}
\item each pair of vertices lying on the same cup or the same
cap is labelled so that one is $\up$ and one is $\down$;
\item each pair of
vertices lying on the same line segment is labelled
so that both are $\up$ or both are $\down$;
\item all other vertices are labelled either $\circ$ or $\times$.
\end{itemize}
Here is an example:
$$
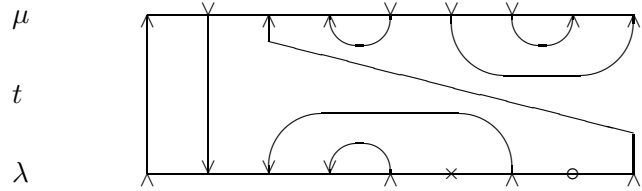
\begin{picture}(30,72)
\put(-116,60){$\mu$}\put(-116,30){$t$}\put(-116,0){$\la$}
\put(-64.8,3){\line(1,0){184}}
\put(-44.7,3.1){$\scriptstyle\down$}
\put(-67.7,-1.6){$\scriptstyle\up$}
\put(70.3,-1.6){$\scriptstyle\up$}
\put(-21.7,3.1){$\scriptstyle\down$}
\put(24.3,-1.6){$\scriptstyle\up$}
\put(1.3,58.4){$\scriptstyle\up$}
\put(93.3,58.4){$\scriptstyle\up$}
\put(116.3,58.4){$\scriptstyle\up$}
\put(1.3,3.1){$\scriptstyle\down$}
\put(24.3,63.1){$\scriptstyle\down$}
\put(47.3,63.1){$\scriptstyle\down$}
\put(116.3,-1.6){$\scriptstyle\up$}
\put(70.3,63.1){$\scriptstyle\down$}
\put(-21.7,58.4){$\scriptstyle\up$}
\put(-44.7,63.1){$\scriptstyle\down$}
\put(-67.7,58.4){$\scriptstyle\up$}
\put(46.8,1){$\scriptstyle\times$}
\put(93.2,.2){$\circ$}
\put(27,3){\oval(92,46)[t]}
\put(15.5,3){\oval(23,23)[t]}
\put(119,3){\line(0,1){15}}
\put(-64.8,63){\line(1,0){184}}
\put(-42,3){\line(0,1){60}}
\put(-65,3){\line(0,1){60}}
\put(84.5,63){\oval(69,46)[b]}
\put(15.5,63){\oval(23,23)[b]}
\put(84.5,63){\oval(23,23)[b]}
\put(-19,63){\line(0,-1){10}}
\put(-19,53){\line(4,-1){138}}
\end{picture}
$$
An {\em oriented $\La\Ga$-circle diagram}
means a composite diagram of the form $a \la t \mu b$
in which
$\la t \mu$ is an oriented $\La\Ga$-matching in the above sense
and $a \la$ and $\mu b$ are oriented
cup and cap diagrams in the sense of \cite[$\S$2]{BS}.

More generally, suppose that we are given a sequence
$\bLa = \La_k \cdots \La_0$ of blocks for some $k \geq 0$.
An {\em oriented $\bLa$-matching} means a
composite diagram
of the form
\begin{equation}\label{imp}
\bt[\bla] := \la_k t_k \la_{k-1} \cdots \la_1 t_1 \la_0
\end{equation}
where $\bla = \la_k\cdots \la_0$ is a sequence of weights
such that
$\la_{i} t_i\la_{i-1}$ is an oriented $\La_i \La_{i-1}$-matching for
each $i=1,\dots,k$.
We stress that {\em every} number
line in the diagram $\bt[\bla]$ is decorated by a weight.
Finally, given in addition a cup diagram $a$ and a cap diagram $b$ such that
$a \lambda_k$ is an oriented cup diagram and
$\lambda_0 b$ is an oriented cap diagram, we can glue
$a$ to the bottom and $b$ to the top of the diagram $\bt[\bla]$
to obtain an {\em oriented $\bLa$-circle diagram}
denoted $a \,\bt[\bla]\,b$.

We say that a $\bLa$-matching $\bt$
is {\em proper} if at least one oriented $\bLa$-matching
exists. To formulate an obvious implication of this
condition, introduce the notation
$\up(\Ga)$ (resp.\ $\down(\Ga)$) to denote the (possibly infinite)
number of vertices of a weight belonging to a
block $\Ga$ that are labelled by the symbol $\up$ (resp.\ $\down$).

\begin{Lemma}\label{updown}
If $\bt = t_k \cdots t_1$ is a proper $\bLa$-matching then
\begin{align*}
\up(\La_i) - \caps(t_i) &= \up(\La_{i-1}) - \cups(t_i) \geq 0,\\
\down(\La_i) - \caps(t_i) &= \down(\La_{i-1}) - \cups(t_i) \geq 0,
\end{align*}
for each $i=1,\dots,k$.
\end{Lemma}

\begin{proof}
As $\bt$ is proper, some oriented $\bLa$-matching
of the form (\ref{imp}) exists.
Now observe that
$\up(\La_i) - \caps(t_i) = \up(\La_{i-1})-\cups(t_i)$
counts the number of line
segments in the diagram $\la_i t_i \la_{i-1}$
that are oriented $\up$. This establishes the first equation,
and the second is similar.
\end{proof}

As in \cite[$\S2$]{BS}, we refer to a circle in an
oriented $\bLa$-circle diagram
as {\em anti-clockwise} (type $1$) or {\em clockwise} (type $x$)
according to whether the leftmost vertices on the circle are labelled
$\down$ or $\up$ (equivalently, the rightmost vertices are labelled
$\up$ or $\down$).
The {\em degree} of a circle or a line in an oriented circle diagram
means its total number of clockwise cups and caps.
Finally, the degree of the oriented circle diagram itself is the sum of the degrees of
its individual circles and lines.

\begin{Lemma}\label{deglem}
The degree of an anti-clockwise circle in an oriented $\bLa$-circle
diagram is one less than the total number of caps (equivalently, cups) that it contains.
The degree of a clockwise circle
is one more than the total number of caps (equivalently, cups)
that it contains.
The degree of a line is equal to the number of caps or the number
of cups that it contains, whichever is the maximum of the two.
\end{Lemma}

\begin{proof}
This is a straightforward extension of
\cite[Lemma 2.1]{BS}.
\end{proof}

\phantomsubsection{Reductions}
Suppose we are given a $\bLa$-matching
$\bt=t_k \cdots t_1$
for some sequence $\bLa = \La_k \cdots\La_0$ of blocks.
We refer to the
circles in the diagram $\bt$ that do not meet the top or bottom number lines as
{\em internal} circles. The
{\em reduction} of $\bt$
means the  $\La_k\La_0$-matching obtained from $\bt$
by removing all its internal circles and all its number lines
except for the top and bottom ones.
Here is an example
in which one internal circle gets removed:
$$
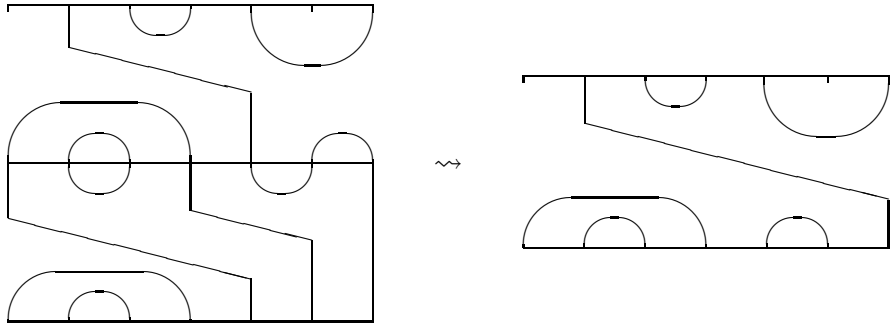
\begin{picture}(195,128)
\put(15.6,123){\oval(23,23)[b]}
\put(-42.2,3){\line(1,0){138.4}}
\put(-42.2,63){\line(1,0){138.4}}
 \put(-42.2,123){\line(1,0){138.4}}
\put(-42,123){\line(0,-1){2.5}}
\put(73,123){\line(0,-1){2.5}}
\put(73,123){\oval(46,46)[b]}

\put(-19,123){\line(0,-1){16}}
\put(50,63){\line(0,1){26.6}}
\put(-18.9,107){\line(4,-1){68.9}}

\put(-7.5,63){\oval(23,23)[b]}
\put(-7.5,63){\oval(23,23)[t]}
\put(-7.5,63){\oval(69,46)[t]}

\put(61.5,63){\oval(23,23)[b]}
\put(84.5,63){\oval(23,23)[t]}

\put(-42,63){\line(0,-1){20.5}}
\put(-41.9,42.2){\line(4,-1){91.7}}
\put(27.1,45.2){\line(4,-1){45.7}}
\put(50,3){\line(0,1){16.3}}
\put(27,63){\line(0,-1){18}}

\put(96,3){\line(0,1){60}}

\put(73,3){\line(0,1){30.7}}

\put(-7.5,3){\oval(69,38)[t]}
\put(-7.5,3){\oval(23,23)[t]}
\end{picture}
\begin{picture}(55,120)
\put(-76,60){$\rightsquigarrow$}
\put(-42,96){\line(1,0){138.3}}
\put(-42,96){\line(0,-1){2.5}}
\put(-42,31){\line(1,0){138.3}}
\put(72.5,96){\oval(47.2,46)[b]}
\put(15.6,96){\oval(23,23)[b]}

\put(-18.8,96){\line(0,-1){18}}
\put(-18.6,78){\line(4,-1){114.6}}
\put(96,31){\line(0,1){18}}

\put(73,96){\line(0,-1){2.5}}
\put(-7.5,31){\oval(23,23)[t]}
\put(61.5,31){\oval(23,23)[t]}
\put(-7.5,31){\oval(69,38)[t]}
\end{picture}
$$

\begin{Lemma}\label{adeg}
Assume that
 $a\,\bt[\bla]\,b$ is an oriented $\bLa$-circle diagram
for some sequence $\bla = \la_k\cdots\la_0$ of weights.
Let $u$ be the reduction of $\bt$. Then
$a \la_k u \la_0 b$ is an oriented $\La_k \La_0$-circle diagram and
\begin{align*}
\deg(a\,\bt[\bla]\,b) &=
\deg(a \la_k u \la_0 b)
+ \caps(t_1)+\cdots+\caps(t_k)-\caps(u)
 + p-q\\
&=\deg(a \la_k u \la_0 b)
+ \cups(t_1)+\cdots+\cups(t_k)-\cups(u)
 + p-q,
\end{align*}
where $p$ (resp.\ $q$) is the number of internal circles
of $\bt$ that are clockwise (resp.\ anti-clockwise)
in the diagram $a\,\bt[\bla]\,b$.
\end{Lemma}

\begin{proof}
Passing from $a \bt b$ to  $a u b$
involves removing all internal circles, each of which
obviously has the same number of cups and caps,
and also
removing an equal number of cups and caps from
each of the other
circles and lines in the diagram.
Moreover the total number of caps removed
is equal to $\caps(t_1)+\cdots+\caps(t_k)-\caps(u)$.
In view of these observations, the lemma follows directly
from Lemma~\ref{deglem}.
\end{proof}

Instead suppose that we are given a pair of blocks $\La$ and $\Ga$
and a $\La \Ga$-matching $t$.
Let $a$ be a cup diagram
whose number line agrees with the bottom number line of $t$.
We refer to circles or lines in the composite
diagram $at$ that do not cross the top number line
as {\em lower} circles or lines.
The {\em lower reduction} of the diagram $at$ means the cup
diagram obtained by removing the bottom number line and
all lower circles and lines.
Here is an example illustrating
such a lower reduction, in which one lower circle and one lower line
get removed:
$$
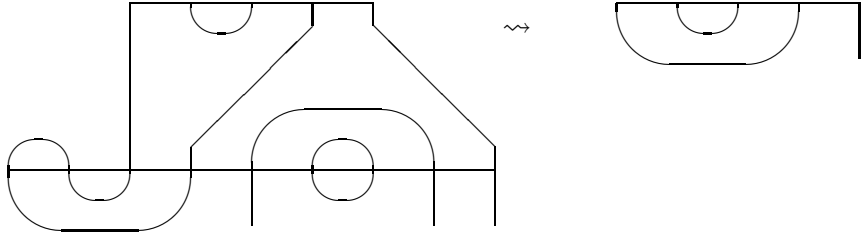
\begin{picture}(200,91)
\put(-19,84){\line(1,0){92}}
\put(15.5,84){\oval(23,23)[b]}
\put(50,84){\line(0,-1){8.5}}
\put(73,84){\line(0,-1){8.6}}
\put(165,84){\line(1,0){92}}
\put(199.5,84){\oval(23,23)[b]}
\put(199.5,84){\oval(69,46)[b]}
\put(257,84){\line(0,-1){21}}

\put(-65,21){\line(1,0){184}}
\put(61.5,21){\oval(23,23)[b]}
\put(-30.5,21){\oval(23,23)[b]}
\put(-30.5,21){\oval(69,46)[b]}
\put(61.5,21){\oval(23,23)[t]}
\put(-53.5,21){\oval(23,23)[t]}
\put(61.5,21){\oval(69,46)[t]}
\put(-19,21){\line(0,1){63}}
\put(4,21){\line(0,1){8.5}}
\put(119,21){\line(0,1){8.8}}
\put(119,21){\line(0,-1){21}}
\put(96,21){\line(0,-1){21}}
\put(27,21){\line(0,-1){21}}

\put(4.2,29.5){\line(1,1){45.8}}
\put(119,29.8){\line(-1,1){45.8}}

\put(122,72){$\rightsquigarrow$}
\end{picture}
$$
Dually,
if $b$ is a cap diagram
whose number line agrees with the top number line of $t$,
the {\em upper} circles and lines
in the diagram $tb$ means the ones that do not cross the bottom
number line.
Then the {\em upper reduction} of $tb$ means the cap
diagram obtained by removing the top number line
and all upper circles and lines.

\begin{Lemma}\label{cdeg}
If $a \la t \mu b$ is an oriented $\La\Ga$-circle diagram
and $c$ is the lower reduction of $at$ then
$c \mu b$ is an oriented circle diagram and
$$
\deg(a \la t \mu b)  = \deg(c \mu b)
+ \caps(t) + p-q,
$$
where $p$ (resp.\ $q$) is the number of lower circles
of $at$ that are clockwise (resp.\ anti-clockwise) in
the diagram $a \la t \mu b$.
For the dual statement about upper reduction, replace $\caps(t)$ by $\cups(t)$.
\end{Lemma}

\begin{proof}
Passing from $atb$ to $cb$ involves removing lower circles,
which obviously
have the same number of caps and cups, and
lower lines, which have one more cap than cup,
and also removing an equal number of cups and caps
from each of the other circles and lines in the diagram.
Moreover the total number of caps removed is equal to $\caps(t)$.
Now apply Lemma~\ref{deglem}.
\end{proof}

\section{Geometric bimodules}\label{smatchings}

In this section we generalise
a construction of Khovanov \cite[$\S$2.7]{K2}
associating a geometric bimodule
to each crossingless matching $t$.
Recall for any
{block} $\La$ of weights that there are corresponding
positively graded associative algebras
$H_\La$ and $K_\La$; see \cite[$\S$3]{BS} for the definition of
$H_\La$ for
Khovanov blocks,
and \cite[$\S$4]{BS} and \cite[$\S$6]{BS} for the definitions of
$K_\La$ and $H_\La$ for general blocks.

\phantomsubsection{\boldmath Khovanov's geometric bimodules}
Assume in this subsection that
$\bLa = \La_k \cdots \La_0$ is a sequence of Khovanov
blocks, so for each $i=0,\dots,k$ the block
$\La_i$ consists of bounded weights with
$\up(\La_i) = \down(\La_i)$. Let $\bt = t_k \cdots t_1$ be a $\bLa$-matching.
Using the notation (\ref{imp}), let $H_{\bLa}^{\bt}$
be the graded vector space on homogeneous basis
\begin{equation}\label{equiv3}
\left\{(a \,\bt[\bla]\, b) \:\big|\:
\text{for all closed oriented $\bLa$-circle diagrams
$a \,\bt[\bla]\,b$}\right\}.
\end{equation}
Note $H^{\bt}_{\bLa}$ is non-zero if and only if
$\bt$ is a proper $\bLa$-matching.
For example, if $k = 0$ then the sequence $\bLa$ is a single block
$\La$, ${\bt}$
is empty, and $H_{\bLa}^{\bt}$ is the graded vector
space underlying Khovanov's algebra $H_{\La}$
as defined in \cite[$\S$3]{BS}.
In general, define a degree preserving linear map
\begin{equation}\label{star}
*:H_{\bLa}^{\bt} \rightarrow H_{\bLa^*}^{\bt^*},\qquad
(a \,\bt[\bla]\, b) \mapsto (b^* \,\bt^*[\bla^*] \,a^*)
\end{equation}
where $\bLa^* := \La_0 \cdots \La_k,
\bla^* := \la_0 \cdots \la_k,
\bt^* := t_1^* \cdots t_k^*$, and $a^*, b^*$ and $t_i^*$
denote the mirror images of $a, b$ and $t_i$ in a horizontal axis.

Let $\bGa = \Ga_l \cdots \Ga_0$
be another sequence of Khovanov blocks such that $\La_0 = \Ga_l$.
We introduce the notation
$\bLa \wr \bGa$ for the concatenated sequence
$\La_k \cdots \La_1 \Ga_l \cdots \Ga_0$; observe here we have
skipped one copy of $\La_0$ in the middle.
Let $\bu = u_l \cdots u_1$ be a $\bGa$-matching and
note that
$\bt \bu := t_k \cdots t_1 u_l\cdots u_1$ is
a $\bLa \wr\bGa$-matching.
We define a degree preserving linear multiplication
\begin{equation}\label{oldmult}
m:\quad H_{\bLa}^{\bt} \otimes H_{\bGa}^{\bu} \rightarrow H_{\bLa \wr\bGa}^{\bt \bu}
\end{equation}
by defining the product
$(a \, \bt[\bla] \,b ) (c \, \bu[\bmu]\, d)$ of two basis
vectors as follows.
If $b^* \neq c$ then we simply declare that
$(a \, \bt[\bla] \,b ) (c \, \bu[\bmu]\, d) = 0$.
If
$b^* = c$ then we draw the oriented circle diagram $a\,\bt[\bla]\, b$
underneath $c \,\bu[\bmu]\, d$
to create a new diagram with a symmetric {middle section} containing the cap
diagram $b$ underneath the cup diagram $c$. Then we perform surgery procedures on
this middle section exactly as in \cite[$\S$3]{BS} to obtain a disjoint
union of diagrams with no cups or caps left in their middle sections. Finally,
collapse the resulting diagrams by identifying the number lines above and below
the middle section to obtain a disjoint union of some new closed oriented
$\bLa\wr\bGa$-circle diagrams of the form
$a \, (\bt\bu)[\bnu] \, d$,
for sequences $\bnu = \nu_{k+l} \cdots \nu_0$
with
$$
\nu_{k+l} \in \La_k,\dots,\nu_{l} \in \La_0,\quad
\nu_l \in \Ga_l,\dots,\nu_0 \in \Ga_0.
$$
Define the desired product
$(a \, \bt[\bla] \,b ) (c \, \bu[\bmu]\, d)$
to be the sum of the basis vectors
of $H_{\bLa\wr\bGa}^{\bt\bu}$ corresponding to these diagrams.

The fact that this multiplication is well defined
follows like in
\cite[$\S$3]{BS}
by reinterpreting the construction in terms of a certain TQFT.
Indeed, in the special case that $k=l=0$
the rule just described is exactly the definition of
the multiplication on Khovanov's
algebra given there.
The multiplication $m$ is associative in the
sense that,
given a third sequence $\bUps$
of Khovanov blocks such that $\Upsilon_0 = \Ga_k$ and an
$\bUps$-matching $\bs$,
the following
diagram commutes:
\begin{equation}\label{assoc}
\begin{CD}
H_{\bUps}^{\bs}\otimes
H_{\bLa}^\bt \otimes H_{\bGa}^{\bu}
&@>1 \otimes m>>&H_{\bUps}^\bs \otimes H_{\bLa \wr \bGa}^{\bt\bu}\\
@V m\otimes 1VV&&@VVmV\\
H_{\bUps\wr \bLa}^{\bs \bt} \otimes H_{\bGa}^{\bu}
&@>>m>&H_{\bUps\wr \bLa\wr \bGa}^{\bs\bt\bu}
\end{CD}
\end{equation}
Again this follows easily from the TQFT
point of view.
Moreover the linear map $*$
from (\ref{star}) is anti-multiplicative in the sense that the following
diagram commutes:
\begin{equation}\label{anti}
\begin{CD}
H_{\bLa}^{\bt} \otimes H_{\bGa}^{\bu}&@>P\circ (* \otimes *)>>&H_{\bGa^*}^{\bu^*}\otimes H_{\bLa^*}^{\bt^*}\\
@Vm VV&&@VVm V\\
H_{\bLa\wr \bGa}^{\bt\bu}&@>>\phantom{asd,}*\phantom{asd}>&H_{\bGa^*\wr\bLa^*}^{\bu^*\bt^*}
\end{CD}
\end{equation}
where $P$ is the flip $x \otimes y \mapsto y \otimes x$.
Finally, the fact that $m$
is a homogeneous linear map of degree zero can be checked directly
using Lemma~\ref{deglem}; see \cite[$\S$3]{BS} once more
for a similar argument.
All the assertions just made can also be extracted from \cite[$\S$2.7]{K2}.

Applying (\ref{assoc})
with $H_{\bUps}^\bs := H_{\La_k}$ and $H_{\bGa}^{\bu} := H_{\La_0}$,
we see that the multiplication $m$ defines
commuting left $H_{\La_k}$-
and right $H_{\La_0}$-actions on $H_{\bLa}^{\bt}$.
Another couple of applications of (\ref{assoc}) shows that these actions are associative.
Finally,
recalling the definitions
of the cup and cap diagrams
 $\underline{\la}$ and $\overline{\la}$
associated to a weight $\la$
from \cite[$\S$2]{BS} and the idempotents
$e_\la$ from
\cite[(3.3)]{BS}, we have that
\begin{align}\label{idempotents1}
e_\al (a\,\bt[\bla]\, b) &= \left\{
\begin{array}{ll}
(a \,\bt[\bla]\, b) &\!\text{if $\underline{\al}=a$,}\\
0&\!\text{otherwise,\!\!}
\end{array}\right.\\
(a\,\bt[\bla]\, b) e_\be &= \left\{
\begin{array}{ll}
(a\,\bt[\bla]\, b) &\!\text{if $b = \overline{\be}$,}\\
0&\!\text{otherwise,\!\!}
\end{array}\right.\label{idempotents2}
\end{align}
for  $\al \in \La_k^\circ, \be \in \La_0^\circ$
and any basis vector $(a\,\bt[\bla]\, b)\in H_{\bLa}^{\bt}$.
This implies that the $H_{\La_k}$- and $H_{\La_0}$-actions
are unital.
Hence we have made $H_{\bLa}^{\bt}$ into a well-defined graded
$(H_{\La_k},H_{\La_0})$-bimodule.
Specialising further,
take $k=1$ so that the sequence $\bLa$ is just a pair of blocks
$\La \Ga$ and
$\bt$ is a single $\La\Ga$-matching $t$.
In this case we usually
write $H_{\La\Ga}^t$ instead of $H_{\bLa}^\bt$,
and we have shown that $H_{\La\Ga}^t$ is a graded $(H_\La,H_\Ga)$-bimodule
which is non-zero if and only if $t$ is a proper $\La\Ga$-matching.
These bimodules are precisely
Khovanov's {\em geometric bimodules} from \cite[$\S$2.7]{K2}.

The following theorem
is a generalisation of \cite[Theorem 3.1]{BS}
and is essential for  all subsequent constructions and
proofs in this article.

\begin{Theorem}\label{cell3}
Let notation be as in (\ref{oldmult}) and suppose we are given basis vectors
 $(a \, \bt[\bla] \, b) \in H_\bLa^\bt$
and $(c \, \bu[\bmu]\, d) \in H_{\bGa}^\bu$.
Write $a \,\bt[\bla]\, b$ as
$\vec{a} \la b$ where $\vec{a} := a \la_k t_k \la_{k-1} \cdots \la_1 t_1$
and $\la := \la_0$.
Similarly write
$c \, \bu[\bmu]\, d =
c \mu \vec{d}$ where
$\mu := \mu_l$
and
$\vec{d} := u_l \mu_{l-1} \cdots \mu_1 u_1 \mu_0 d$.
Then the multiplication satisfies
$$
(\vec{a} \la b) (c \mu \vec{d}) = \left\{
\begin{array}{ll}
0&\text{if $b \neq c^*$,}\\
s_{\vec{a} \la b}(\mu) (\vec{a} \mu \vec{d}) + (\dagger)
&\text{if $b = c^*$ and $\vec{a} \mu$ is oriented,}\\
(\dagger)&\text{otherwise,}
\end{array}\right.
$$
where
\begin{itemize}
\item[(i)]
$(\dagger)$ denotes a linear combination of basis vectors of $H_{\bLa\wr\bGa}^{\bt\bu}$
of the form $(a \,(\bt\bu)[\bnu]\,d)$
for
$\bnu = \nu_{k+l}\cdots \nu_0$ with
$\nu_l > \mu_l, \nu_{l-1}\geq\mu_{l-1},\dots,\nu_0 \geq \mu_0$;
\item[(ii)]
the scalar $s_{\vec{a} \la b}(\mu)\in \{0,1\}$ depends only on $\vec{a} \la b$ and $\mu$
(but not on $\vec{d}$);
\item[(iii)]
if $b = \overline{\la} = c^*$ and $\vec{a} \mu$ is oriented then
$s_{\vec{a} \la  b}(\mu) = 1$;
\item[(iv)] if $k = 0$ then the $s_{\vec{a} \la b}(\mu)$ is equal to the
scalar $s_{a \la b}(\mu)$ from \cite[Theorem 3.1]{BS}.
\end{itemize}
\end{Theorem}

\begin{proof}
Mimic the proof of \cite[Theorem 3.1]{BS} replacing $a$ by $\vec{a}$ and $d$ by
$\vec{d}$.
\end{proof}

\begin{Corollary}\label{ideal3}
The product $(a\,\bt[\bla]\, b)(c\,\bu[\bmu]\, d)$
of any pair of basis vectors of $H_{\bLa}^{\bt}$ and
$H_{\bGa}^{\bu}$
is a linear combination of basis vectors of $H_{\bLa\wr\bGa}^{\bt\bu}$
of the form $(a\,(\bt\bu)[\bnu]\,d)$ for
$\bnu = \nu_{k+l} \cdots \nu_0$ with
$\la_k \leq \nu_{k+l},\dots,\la_0 \leq \nu_l \geq  \mu_l,\dots,
\nu_0 \geq \mu_0$.
\end{Corollary}

\begin{proof}
Theorem~\ref{cell3}(i) implies that the product is a linear combination
of vectors $(a\,(\bt\bu)[\bnu]\,d)$
for $\bnu = \nu_{k+l}\cdots\nu_0$ such that
$\nu_l \geq \mu_l,\dots,\nu_0 \geq \mu_0$.
To show also that $\la_k \leq \nu_{k+l},\dots,\la_0 \leq \nu_l$,
use the anti-multiplicativity of the map $*$ from (\ref{anti}) and
argue as in the proof of \cite[Corollary~3.2]{BS}.
\end{proof}

\phantomsubsection{\boldmath Geometric bimodules for $K_\La$}
Now we construct the analogous geometric
bimodules for the algebras
$K_\La$.
Assume for this that we are
given a sequence $\bLa=\La_k\cdots \La_0$ of blocks
for some $k \geq 0$ and a
$\bLa$-matching $\bt=t_k \cdots t_1$, dropping the assumption that the
$\La_i$ are Khovanov blocks.
Let $K_{\bLa}^{\bt}$ be the graded vector space
on homogeneous basis
\begin{equation}\label{equiv4}
\left\{
(a \, \bt[\bla]\,b)\:\big|\:
\text{for all oriented $\bLa$-circle diagrams
$a\, \bt[\bla]\,b$}\right\}.
\end{equation}
Again, $K_\bLa^\bt$ is non-zero if and only if $\bt$ is proper.
For example, if $k=0$ then $\bLa$ is a single block $\La$, $\bt$ is empty,
and $K^\bt_\bLa$ is the graded vector space
underlying the algebra $K_\La$ from \cite[$\S$4]{BS}.
In general, we define a degree preserving linear map
\begin{equation}\label{star2}
*:K_{\bLa}^{\bt} \rightarrow K_{\bLa^*}^{\bt^*},\qquad
(a \,\bt[\bla]\, b) \mapsto (b^* \,\bt^*[\bla^*]\, a^*),
\end{equation}
notation as in (\ref{star}).

We want to extend the multiplication (\ref{oldmult})
to this new setting. In practise, this multiplication is best computed
by following exactly the same steps used to compute
(\ref{oldmult}), but replacing the
surgery procedure
by the generalised surgery procedure from
\cite[$\S$6]{BS}. However this approach does not easily imply that
the bimodule multiplication is associative,
so instead we proceed in a more roundabout manner mimicing the
definition of multiplication in $K_\La$ as formulated in \cite[$\S$4]{BS}.
Before we can do this, we must generalise the definitions of
closure and extension from \cite[$\S$4]{BS}
to incorporate crossingless matchings.

\phantomsubsection{Closures and extensions}
To generalise closure, assume that $\bLa$ is a sequence of
blocks of bounded weights and that $\bt=t_k\cdots t_{1}$ is a
proper $\bLa$-matching.
Let $p, q \geq 0$ be integers such that
$p - \ups(\La_i) = q - \downs(\La_i) \geq 0$
for each $i$; this is possible because the difference
$\ups(\La_i) - \downs(\La_i)$ is independent of $i$ by Lemma~\ref{updown}.
Using the
notation \cite[(4.4)]{BS}
but working always now with these new choices for
$p$ and $q$, let $\De_i$ be the block
generated by $\cl(\La_i)$
and set $\bDe := \De_k\cdots\De_0$.
Also set $\cl(\bt) :=
\cl(t_k) \cdots \cl(t_1)$, where the {\em closure} $\cl(t)$
of a crossingless matching $t$
is obtained by adding $p$ new vertices to the left ends and $q$ new vertices to
the right hand ends of all the number lines in $t$, then joining these new
vertices together in order by some new vertical line segments.
Finally, for an oriented $\bLa$-circle diagram
$(a\,\bt[\bla]\,b)$, define its closure
\begin{equation*}
\cl(a\,\bt[\bla]\,b) :=
\cl(a) \,\cl(\bt) [\cl(\la_k) \cdots \cl(\la_0)] \,\cl(b)
\end{equation*}
where  $\cl(a)$ and $\cl(b)$ are as in \cite[(4.5)--(4.6)]{BS}.
In the same spirit as \cite[Lemma~4.2]{BS},
this map defines a bijection between the
oriented $\bLa$-circle diagrams
indexing the basis of $K_{\bLa}^{\bt}$ and certain closed
oriented $\bDe$-circle diagrams indexing part of the
basis of $H_{\bDe}^{\cl(\bt)}$.
More precisely, let $I_{\bLa}$ denote the subspace of $H_{\bDe}^{\cl(\bt)}$
spanned by all basis vectors $(a\,\cl(\bt)[\bla]\,b)$
such that $\bla = \la_k\cdots\la_0$
with
$\la_i \notin \cl(\La_i)$
for at least one $0 \leq i \leq k$. Then the map
\begin{align}\label{clmap}
\cl:K_{\bLa}^{\bt} &\stackrel{\sim}{\rightarrow} H_{\bDe}^{\cl(\bt)}/ I_{\bLa},
\qquad
(a \,\bt[\bla]\,b) \mapsto
(\cl(a \,\bt[\bla]\,b)) + I_\bLa
\end{align}
is a graded vector
space isomorphism.

To generalise extension, given crossingless matchings $s$ and $t$, we
write $s
\prec t$ to indicate that $t$ can be obtained from $s$ by extending each of its
number lines then adding possibly infinitely many more vertical line segments.
Now assume that we are given two sequences of arbitrary blocks,
$\bLa=\La_k\cdots\La_0$ and $\bUps=\Upsilon_k\cdots\Upsilon_0$, such that $\Upsilon_i\prec \La_i$ for each $i$ in the sense explained just before \cite[(4.9)]{BS}.
Let
$\bs = s_k \cdots s_1$ be a proper $\bUps$-matching
and $\bt = t_k \cdots t_1$ be a proper
$\bLa$-matching, such that
$s_i \prec t_i$ for each $i$.
Extending the definition \cite[(4.11)]{BS} to
oriented $\bUps$-circle diagrams, we
set
\begin{equation*}
\ex^\bLa_\bUps(a \, \bs[\bla]\,b) :=
\ex_{\Upsilon_k}^{\La_k}(a) \,\bt[\ex^{\La_k}_{\Upsilon_k}(\la_k) \cdots
\ex^{\La_0}_{\Upsilon_0}(\la_0)]\,
\ex_{\Upsilon_0}^{\La_0}(b).
\end{equation*}
Then the map
\begin{equation}\label{exxx}
\ex_{\bUps}^{\bLa}:K_{\bUps}^{\bs} \hookrightarrow K_{\bLa}^{\bt},
\qquad (a \,\bs[\bla]\,b)
\mapsto (\ex_{\bUps}^\bLa(a\,\bs[\bla]\,b))
\end{equation}
is a degree preserving injective linear map.

\phantomsubsection{\boldmath The bimodule multiplication}
Finally assume $\bLa = \La_k\cdots\La_0$ and
$\bGa = \Ga_l\cdots\Ga_0$ are sequences of arbitrary blocks
with $\La_0 = \Ga_l$, and
$\bt = t_k\cdots t_1$
and $\bu = u_l \cdots u_1$ are
$\bLa$-
and $\bGa$-matchings.
We are ready to define
a degree preserving linear multiplication
\begin{equation}
\label{newmult}
m:K_{\bLa}^{\bt}
\otimes
K_{\bGa}^{\bu}
\rightarrow
K_{\bLa\wr\bGa}^{\bt\bu}.
\end{equation}
If either $\bt$ is not a proper $\bLa$-matching
or $\bu$ is not a proper $\bGa$-matching then
$K^\bt_\bLa\otimes K^\bu_\bGa = \{0\}$, and we simply have to take $m := 0$.
Assume from now on that both $\bt$ and $\bu$ are both proper.

Suppose to start with that each $\La_i$ and each $\Ga_j$
consists of bounded weights.
Applying Lemma~\ref{updown}, we can pick $p,q \geq 0$
so that
$p - \ups(\La_i) = q - \downs(\La_i) \geq 0$
and $p - \ups(\Ga_j) = q - \downs(\Ga_j) \geq 0$ for all $i,j$;
it is easy to see that the construction we are about to explain
is independent of the particular choices of $p,q$ just made.
Let $\bDe := \De_k\cdots\De_0$ where $\De_i$ is the block generated
by $\cl(\La_i)$ and $\bPi := \Pi_l\cdots\Pi_0$ where
$\Pi_i$ is the block generated by $\cl(\Ga_i)$.
Noting also
that $\cl(\bt)\cl(\bu) = \cl(\bt\bu)$,
the isomorphism (\ref{clmap}) gives the following three graded vector space isomorphisms:
\begin{equation}\label{cl3}
\cl:K^\bt_\bLa \stackrel{\sim}{\rightarrow}
H_{\bDe}^{\cl(\bt)} \big/ I_{\bLa},
\quad
K^\bu_\bGa \stackrel{\sim}{\rightarrow}
H_{\bPi}^{\cl(\bu)} \big/ I_{\bGa},
\quad
K^{\bt\bu}_{\bLa\wr\bGa} \stackrel{\sim}{\rightarrow}
H_{\bDe\wr\bPi}^{\cl(\bt\bu)} \big/ I_{\bLa\wr\bGa}.
\end{equation}
By \cite[Lemma 4.1]{BS} and Corollary~\ref{ideal3}, the multiplication map
$$
H_{\bDe}^{\cl(\bt)} \otimes H_{\bPi}^{\cl(\bu)} \rightarrow H_{\bDe\wr\bPi}^{\cl(\bt\bu)}
$$
from
(\ref{oldmult}) has the properties that
$I_{\bLa} H_{\bPi}^{\cl(\bu)}
 \subseteq I_{\bLa\wr\bGa}$
and $H_{\bDe}^{\cl(\bt)}
 I_{\bGa}
 \subseteq I_{\bLa\wr\bGa}$.
Hence it factors through the quotients to give a well-defined degree preserving
linear map
$$
\left(H_{\bDe}^{\cl(\bt)} \big/ I_{\bLa}\right) \otimes \left(H_{\bPi}^{\cl(\bu)} \big/
I_{\bGa}\right) \rightarrow H_{\bDe\wr\bPi}^{\cl(\bt\bu)} \big/ I_{\bLa\wr\bGa}.
$$
Transporting this through the
isomorphisms $\cl$ from (\ref{cl3}), we obtain
the desired multiplication (\ref{newmult}).

It remains to drop the assumption that each $\La_i$ and $\Ga_j$
consists of bounded weights;
formally this involves taking direct limits as explained immediately after
\cite[Lemma 4.3]{BS}
but we will suppress the details.
Take vectors $x \in K_{\bLa}^{\bt}$ and $y \in K_{\bGa}^{\bu}$.
Recalling (\ref{exxx}), choose sequences of blocks
$\bUps = \Upsilon_k\cdots\Upsilon_0$
and
$\bOm = \Omega_l\cdots\Omega_0$ and
proper $\bUps$- and $\bOm$-matchings
$\br = r_k\cdots r_1$
 and
$\bs = s_l\cdots s_1$ such that
\begin{itemize}
\item $\Upsilon_i \prec \La_i$ and
$r_i \prec t_i$ for each $i$,
and $\Omega_j \prec \Ga_j$ and $s_j \prec u_j$ for each $j$;
\item $x$ is in the image of the map
$\ex_\bUps^\bLa: K_\bUps^\br \hookrightarrow K_{\bLa}^\bt$,
and $y$ is in the image of the map
$\ex_\bOm^\bGa:K_\bOm^\bs \hookrightarrow K_\bGa^\bu$.
\end{itemize}
Then we truncate by applying the inverse maps to $\ex_{\bUps}^{\bLa}$ and
$\ex_{\bOm}^{\bGa}$ to $x$ and $y$, respectively, compute the resulting
bounded product as in the previous paragraph
to obtain an element of
$K_{\bUps\wr\bOm}^{\br \bs}$,
then finally apply the map $\ex_{\bUps\wr\bOm}^{\bLa\wr\bGa}$
to get the desired product $xy \in K_{\bLa\wr\bGa}^{\bt\bu}$. This
completes the definition of the multiplication (\ref{newmult}) in general.

Using the associativity from (\ref{assoc}),
one checks that the new multiplication is also associative
in the same sense.
Also, the map $*$ from (\ref{star2}) is
anti-multiplicative in the same sense as (\ref{anti}).
Finally (\ref{idempotents1})--(\ref{idempotents2})
remain true in the new setting
for arbitrary $\alpha \in \La_k, \beta \in \La_0$
and any $(a \,\bt[\bla]\,b) \in K^\bt_\bLa$.
So as before $m$ makes $K^\bt_\bLa$ into a well-defined
graded $(K_{\La_k},K_{\La_0})$-bimodule.
Generalising the decomposition \cite[(4.13)]{BS}, we also have that
\begin{equation}\label{stiller}
K_{\bLa}^{\bt} = \bigoplus_{\alpha \in \La_r, \beta \in \La_0} e_\alpha
K_{\bLa}^{\bt} e_\beta.
\end{equation}
By considering the explicit parametrisation of the basis
using
(a slightly modified version of)
\cite[Lemma~2.4(i)]{BS}, it follows that each summand on the right hand side of
(\ref{stiller}) is finite dimensional. Moreover $K_{\bLa}^{\bt}$ itself is
locally finite dimensional and bounded below.

The following theorem is a generalisation of \cite[Theorem 4.4]{BS}.
We stress that its statement is
{\em exactly the same} as the statement of Theorem~\ref{cell3},
just with $H$ replaced by $K$ everywhere.

\begin{Theorem}\label{cell4}
Let notation be as in (\ref{newmult}) and suppose we are given basis vectors
 $(a \, \bt[\bla] \, b) \in K_\bLa^\bt$
and $(c \, \bu[\bmu]\, d) \in K_{\bGa}^\bu$.
Write $a \,\bt[\bla]\, b$ as
$\vec{a} \la b$ where $\vec{a} := a \la_k t_k \la_{k-1} \cdots \la_1 t_1$
and $\la := \la_0$.
Similarly write
$c \, \bu[\bmu]\, d =
c \mu \vec{d}$ where
$\mu := \mu_l$
and
$\vec{d} := u_l \mu_{l-1} \cdots \mu_1 u_1 \mu_0 d$.
Then the multiplication satisfies
$$
(\vec{a} \la b) (c \mu \vec{d}) = \left\{
\begin{array}{ll}
0&\text{if $b \neq c^*$,}\\
s_{\vec{a} \la b}(\mu) (\vec{a} \mu \vec{d}) + (\dagger)
&\text{if $b = c^*$ and $\vec{a} \mu$ is oriented,}\\
(\dagger)&\text{otherwise,}
\end{array}\right.
$$
where
\begin{itemize}
\item[(i)]
$(\dagger)$ denotes a linear combination of basis vectors of $K_{\bLa\wr\bGa}^{\bt\bu}$
of the form $(a \,(\bt\bu)[\bnu]\,d)$
for
$\bnu = \nu_{k+l}\cdots \nu_0$ with
$\nu_l > \mu_l, \nu_{l-1}\geq\mu_{l-1},\dots,\nu_0 \geq \mu_0$;
\item[(ii)]
the scalar $s_{\vec{a} \la b}(\mu)\in \{0,1\}$ depends only on $\vec{a} \la b$ and $\mu$
(but not on $\vec{d}$);
\item[(iii)]
if $b = \overline{\la} = c^*$ and $\vec{a} \mu$ is oriented then
$s_{\vec{a} \la  b}(\mu) = 1$;
\item[(iv)] if $k = 0$ then the $s_{\vec{a} \la b}(\mu)$ is equal to the
scalar $s_{a \la b}(\mu)$ from \cite[Theorem 4.4]{BS}.
\end{itemize}
\end{Theorem}

\begin{proof}
This follows from Theorem~\ref{cell3} above in the same fashion
as \cite[Theorem 4.4]{BS} follows from \cite[Theorem 3.1]{BS}.
\end{proof}

\begin{Corollary}\label{ideal4}
The product $(a\,\bt[\bla]\, b)(c\,\bu[\bmu]\, d)$
of any pair of basis vectors of $K_{\bLa}^{\bt}$ and
$K_{\bGa}^{\bu}$
is a linear combination of basis vectors of $K_{\bLa\wr\bGa}^{\bt\bu}$
of the form $(a\,(\bt\bu)[\bnu]\,d)$ for
$\bnu = \nu_{k+l} \cdots \nu_0$ with
$\la_k \leq \nu_{k+l},\dots,\la_0 \leq \nu_l \geq  \mu_l,\dots,
\nu_0 \geq \mu_0$.
\end{Corollary}

\begin{proof}
Argue as in the proof of Corollary~\ref{ideal3}
using the analogue of (\ref{anti}).
\end{proof}

\phantomsubsection{Composition of geometric bimodules}
The next
theorem is crucial and is an extension of \cite[Theorem 1]{K2}.

\begin{Theorem}\label{t1}
Let notation be as in (\ref{newmult}),
so the tensor product
$K_{\bLa}^{\bt} \otimes_{K_{\La_0}}
K_{\bGa}^{\bu}$ is a well-defined graded
$(K_{\La_k},K_{\Ga_0})$-bimodule.
\begin{enumerate}
\item[(i)] Any vector
$(a\,\bt[\bla]\, b)
\otimes(c\,\bu[\bmu]\, d)
\in
K_{\bLa}^{\bt} \otimes_{K_{\La_0}}
K_{\bGa}^{\bu}$
is a linear combination of vectors of the form
\begin{equation}\label{vectors}
(a\,\bt[\nu_{k+l}\cdots\nu_l]\,\overline{\nu}_l)
\otimes (\underline{\nu}_l\,\bu[\nu_l\cdots\nu_0]\, d)
\end{equation}
for $\bnu = \nu_{k+l}\cdots\nu_0$
such that
$a\, (\bt\bu)[\bnu]\,d$ is an oriented
$\bLa\wr\bGa$-circle diagram and
$\la_k \leq \nu_{k+l},\dots,\la_0 \leq \nu_l\geq \mu_l,\dots,\nu_0\geq \mu_0$.
\item[(ii)]
The vectors (\ref{vectors})
for {\em all}
oriented $\bLa\wr\bGa$-circle diagrams
$a \,(\bt\bu)[\bnu]\,d$
form a basis for
$K_{\bLa}^{\bt} \otimes_{K_{\La_0}} K_{\bGa}^{\bu}$.
\item[(iii)]
The multiplication $m$ from (\ref{newmult})
induces an isomorphism
$$
\overline{m}: K_{\bLa}^{\bt} \otimes_{K_{\La_0}} K_{\bGa}^{\bu}
\stackrel{\sim}{\rightarrow} K_{\bLa\wr\bGa}^{\bt\bu}.
$$
of graded $(K_{\La_{k}}, K_{\Ga_{0}})$-bimodules.
\end{enumerate}
\end{Theorem}

\begin{proof}
(i)
Assume for a contradiction that (i) is false.
Fix cup and cap diagrams $a$ and $d$
such that the statement of (i) is wrong
for some $\bla,\bmu,b$ and $c$.
Let
$$
S := \left\{(\bla,\bmu)
\:\Bigg|\:
\begin{array}{l}
\text{$\bla = \la_k\cdots \la_0$ with each $\la_i \in \La_i$,}\\
\text{$\bmu = \mu_l\cdots \mu_0$ with each $\mu_j \in \Ga_j$,}\\
\text{$a\,\bt[\bla]$ and $\bu[\bmu]\,d$ are oriented}
\end{array}
\right\}
$$
Put a partial order on $S$
by declaring that
$(\bla,\bmu) \leq (\bla',\bmu')$ if
$\la_i \leq \la_i'$ and $\mu_j \leq \mu_j'$ for all $i$ and $j$.
For $(\bla,\bmu) \in S$,
let $K(\bla,\bmu)$ denote the span of all the vectors of the form
(\ref{vectors})
for $\bnu = \nu_{k+l}\cdots\nu_0$ such that $a \,(\bt\bu)[\bnu]\,d$
is an oriented $\bLa\wr\bGa$-circle diagram
and $\la_k \leq \nu_{k+l},\dots,\la_0 \leq \nu_l \geq \mu_l,\dots,\nu_0 \geq \mu_0$.
Because the set $S$ is necessarily finite (see \cite[Lemma 2.4(i)]{BS}),
we can choose $(\bla,\bmu) \in S$ maximal
such that
$$
(a \,\bt[\bla]\,b) \otimes (c\,\bu[\bmu]\,d) \notin K(\bla,\bmu)
$$
for some $b$ and $c$.
To get a contradiction from this, we
now split into two cases according to whether
$\la_0 \not< \mu_l$ or $\la_0 \not > \mu_l$.

Suppose to start with that $\la_0 \not< \mu_l$.
By Theorem~\ref{cell4}(iii) and Corollary~\ref{ideal4}, we have that
$$
(a\,\bt[\bla]\,\overline{\la}_0) (\underline{\la}_0 \la_0 b)=
(a\,\bt[\bla]\, b)
 + (\dagger)
$$
where $(\dagger)$ is a linear combination of
$(a\,\bt[\bla']\,b)$'s with $(\bla,\bmu) < (\bla',\bmu)$.
By maximality of $(\bla,\bmu)$, we deduce that
$(\dagger) \otimes (c\,\bu[\bmu]\,d)$ is contained in $K(\bla',\bmu)$,
hence in $K(\bla,\bmu)$ too.
Hence
$$
(a \,\bt[\bla]\,\overline{\la}_0)
(\underline{\la}_0\la_0 b)
\otimes (c\,\bu[\bmu]\,d)
=
(a \,\bt[\bla]\,\overline{\la}_0)
\otimes
(\underline{\la}_0\la_0 b)
(c\,\bu[\bmu]\,d) \notin K(\bla,\bmu).
$$
By Theorem~\ref{cell4}(i)--(ii), we can expand
$$
(\underline{\la}_0 \la_0 b) (c\,\bu[\bmu]\,d) =
(\underline{\la}_0\, \bu[\bmu]\,d) + (\dagger\dagger)
$$
where the first term on the right hand side should be omitted unless $b = c^*$,
$\underline{\la}_0 \mu_l$ is oriented and
$s_{\underline{\la}_0 \la_0 b}(\mu_l) = 1$, and
$(\dagger\dagger)$ denotes a linear combination of
$(\underline{\la}_0\,\bu[\bmu']\,d)$'s with
$(\bla,\bmu) < (\bla,\bmu')$.
By maximality again,
$(a\,\bt(\bla)\,\overline{\la}_0)
\otimes(\dagger\dagger)$ is contained in $K(\bla,\bmu)$
Hence the first term on the right hand side must be present and we have proved that
$$
(\ddagger) :=
(a\,\bt[\bla]\,\overline{\la}_0)
\otimes
(\underline{\la}_0 \,\bu[\bmu]\,d)
\notin K(\bla,\bmu).
$$
This shows in particular that $\underline{\la}_0 \mu_l$
is oriented, hence $\la_0 \leq \mu_l$ by \cite[Lemma~2.3]{BS}.
Because $\la_0 \not < \mu_l$
we deduce that $\la_0 = \mu_l$.
Hence $(\ddagger)$ is equal to the vector
(\ref{vectors}) for
$\bnu := \la_k\cdots\la_1\mu_l\cdots\mu_0$.
So we also have $(\ddagger) \in K(\bla,\bmu)$, contradiction.

Assume instead that $\la_0 \not > \mu_l$.
Then we can apply the anti-multiplicative map $*$ from (\ref{star2})
to reduce to exactly the situation of the previous paragraph,
hence get a similar contradiction.

(ii), (iii)
By the associativity of multiplication,
the images of $xh \otimes y$ and $x \otimes hy$ are equal
for $x \in K_{\bLa}^{\bt}$,
$h \in K_{\La_0}$ and
$y \in K_{\bGa}^{\bu}$.
Hence the multiplication map
$m$ factors through the quotient to
induce a well-defined graded
bimodule homomorphism $\overline{m}$
as in (iii).
To prove that it is an isomorphism, it suffices to show that the restriction
$$
\overline{m}:
e_\alpha K_{\bLa}^{\bt}
\otimes_{K_{\La_0}}
K_{\bGa}^{\bu} e_\beta
\rightarrow
e_\alpha K_{\bLa\wr\bGa}^{\bt\bu}e_\beta
$$
is a vector space isomorphism
for each
fixed $\alpha \in \La_{k}$ and $\beta \in \Ga_{0}$.
Note the vector space on the
right hand side is finite dimensional, with basis given by the
vectors
$y(\bnu) := (\underline{\alpha}\,(\bt\bu)[\bnu]\,\overline{\beta})$
for all $\bnu$ such that $\underline{\alpha}\,(\bt\bu)[\bnu]\,
\overline{\beta}$ is an oriented $\bLa\wr\bGa$-circle diagram.
In view of (i), we know already that
the vectors
$x(\bnu) := (\underline{\alpha}\,\bt[\nu_{k+l}\cdots\nu_l]\,
\overline{\nu}_l)
\otimes
(\underline{\nu}_l\,\bu[\nu_l\cdots\nu_0]\,\overline{\beta})$
indexed by the same set of $\bnu$'s span
$e_\alpha K_{\bLa}^{\bt}
\otimes_{K_{\La_0}}
K_{\bGa}^{\bu} e_\beta$.
Moreover, $\overline{m}(x(\bnu))$ is equal to $y(\bnu)$ plus higher terms
by Theorem~\ref{cell4}(iii) and Corollary~\ref{ideal4}, so $\overline{m}$
is surjective.
Combining these two statements we deduce that $\overline{m}$ is actually
an isomorphism giving (iii), and at the same time
we see that
the $x(\bnu)$'s form a basis giving (ii).
\end{proof}

\phantomsubsection{Reduction of geometric bimodules}
The following theorem
reduces the study of the bimodules $K_{\bLa}^{\bt}$ for
arbitrary sequences $\bLa = \La_k\cdots\La_0$
and $\bt = t_k\cdots t_1$
just to the bimodules of the form $K_{\La \Ga}^t$ for
fixed blocks $\La$ and $\Ga$ and a single $\La \Ga$-matching $t$,
greatly simplifying all our notation in the remainder of the article.
Let $R$ denote the algebra $\C[x] /
(x^2)$. This is a Frobenius algebra with counit
\begin{equation}\label{ntau}
\tau:R \rightarrow \C, \qquad 1\mapsto 0, \quad x\mapsto 1.
\end{equation}
We view $R$ as a graded vector space so that $1$ is in degree $-1$ and $x$ is in degree one. The multiplication and comultiplication are then homogeneous of degree one.

\begin{Theorem}\label{t2}
Suppose that $\bLa = \La_k\cdots\La_0$ is a sequence of blocks and
$\bt = t_k\cdots t_1$ is a proper $\bLa$-matching.
Let $u$ be the reduction of $\bt$ and let
$n$ be the number of internal circles removed in the reduction process.
Then we have that
\begin{align*}
K_{\bLa}^{\bt}
&\cong K_{\La_k \La_{0}}^{u} \otimes R^{\otimes n}
\langle \caps(t_1)+\cdots+\caps(t_k)
-\caps(u)\rangle\\
&= K_{\La_k \La_{0}}^{u} \otimes R^{\otimes n}
\langle \cups(t_1)+\cdots+\cups(t_k)
-\cups(u)\rangle
\end{align*}
as graded $(K_{\La_k}, K_{\La_{0}})$-bimodules (viewing
$K_{\La_k \La_{0}}^{u} \otimes R^{\otimes n}$ as a bimodule via the
natural action on the first tensor factor).
\end{Theorem}

\begin{proof}
Enumerating the $n$ internal circles
in the diagram $\bt$ in some fixed order,
we define a linear map
$$
f:K_{\bLa}^{\bt} \rightarrow K_{\La_k \La_0}^{u}
\otimes R^{\otimes n},\quad
(a\,\bt[\bla]\,b) \mapsto (a \la_k u \la_0 b)
\otimes x_1 \otimes \cdots \otimes x_n
$$
where each $x_i$ is $1$ or $x$ according to whether the $i$th internal circle
of $\bt$ is anti-clockwise or clockwise
in the diagram $a\,\bt[\bla]\,b$. This map is
obviously a bijection. Moreover,
because the internal circles play no role in
the bimodule structure, it is a $(K_{\La_r}, K_{\La_0})$-bimodule homomorphism.
Finally Lemma~\ref{adeg} implies that $f$ is a homogeneous
linear map of degree
$-\caps(t_1)-\cdots-\caps(t_r)+\caps(u)
=-\cups(t_1)-\cdots-\cups(t_r)+\cups(u)$.
This accounts for the degree shift in the statement of the theorem.
\end{proof}

\phantomsubsection{\boldmath Geometric bimodules for $H_\La$}
For an arbitrary block $\La$, recall that the generalised
Khovanov algebra is the subalgebra
\begin{equation}\label{hla}
H_\La := \bigoplus_{\alpha,\beta \in \La^\circ} e_\alpha K_\La e_\beta
\end{equation}
of $K_\La$,
where $\La^\circ$ is the subset of $\La$ consisting of all weights of maximal
defect. Given another block $\Ga$ and a $\La\Ga$-matching $t$,
we define
\begin{equation}
H^t_{\La\Ga} := \bigoplus_{\alpha \in \La^\circ, \beta \in \Ga^\circ}
e_\alpha K_{\La\Ga}^t e_\beta.
\end{equation}
This is naturally a graded $(H_\La,H_\Ga)$-bimodule,
and if $\La$ and $\Ga$ are Khovanov blocks it is the
same as Khovanov's geometric bimodule $H^t_{\La\Ga}$ defined earlier.
We refer to the $H^t_{\La\Ga}$'s as {\em geometric bimodules} for arbitrary
blocks $\La$ and $\Ga$ too.
We are not go to
say anything else about these geometric bimodules in the rest
of the article, since we are mainly interested in
the $K^t_{\La\Ga}$'s. However using
the truncation functor
\begin{equation}\label{efunc}
e:\mod{K_\La} \rightarrow \mod{H_\La},
\quad
M \mapsto \bigoplus_{\la \in \La^\circ} e_\la M
\end{equation}
from \cite[(6.13)]{BS}
it is usually a
routine exercise to deduce analogues for the $H^t_{\La\Ga}$'s
of all our subsequent results about the $K^t_{\La\Ga}$'s.
In doing this, it is useful to note that there is an isomorphism
\begin{equation}\label{us}
e_\La \circ K_{\La \Ga}^t \otimes_{K_\Ga} ?
\cong H_{\La \Ga}^t \otimes_{H_\Ga} ?
\circ e_\Ga
\end{equation}
of functors
from $\mod{K_\Ga}$ to $\mod{H_\La}$, where we have added the subscripts $\La$ and $\Ga$ to the $e$'s
to avoid confusion.
Since we do not need any of these results here, we omit the details.
The main point to prove (\ref{us}) is to observe that all indecomposable summands of
$e_\La K_{\La \Ga}^t$ as a right $K_\Ga$-module are
also summands of $e_\Ga K_\Ga$ (up to degree shifts),
which is a consequence of Theorem~\ref{pf} below.

\section{Projective functors}

Fix blocks $\La$ and $\Ga$ throughout the section.
We are going to study projective functors that arise
by tensoring with geometric bimodules.

\phantomsubsection{Projective functors are exact}
Given a proper $\La\Ga$-matching $t$,
tensoring with the geometric bimodule $K_{\La\Ga}^t$
defines a functor
\begin{equation}\label{pfun}
G_{\La\Ga}^t := K_{\La\Ga}^t \langle -\caps(t)\rangle \otimes_{K_\Ga} ?
:\mod{K_\Ga} \rightarrow \mod{K_\La}
\end{equation}
between the graded module categories.
The extra degree shift by $-\caps(t)$ in this definition
is included to ensure that $G_{\La\Ga}^t$ commutes with
duality; see Theorem~\ref{duality} below.
We will call any functor that is isomorphic to a
finite direct sum of such functors (possibly shifted in degree)
a {\em projective functor}.
Note Theorems~\ref{t1}(iii) and \ref{t2} imply that the composition of two projective functors is again a projective functor.
Moreover, in view of the next lemma, projective functors
corresponding to ``identity matchings''
are equivalences of categories.

\begin{Lemma}\label{straight}
If $t$ is a proper $\La\Ga$-matching
containing no cups or caps
then the functor
$G_{\La\Ga}^t:
\mod{K_\Ga} \rightarrow \mod{K_\La}$
is an equivalence of categories.
\end{Lemma}

\begin{proof}
The matching $t$ determines an order-preserving bijection
$f:\La \rightarrow \Ga$, $\la \mapsto \ga$
such that $\underline{\ga}$ is the lower reduction of
$\underline{\la} t$.
Moreover the induced map
$$
f:K_\La {\rightarrow} K_\Ga,\qquad
(\underline{\alpha} \la \overline{\beta})
\mapsto (\underline{f(\alpha)} f(\la) \overline{f(\beta)})
$$
is an isomorphism of graded algebras.
Identifying $K_\La$ with $K_\Ga$ in this way,
the $(K_\La, K_\Ga)$-bimodule $K_{\La \Ga}^t$
is isomorphic to the regular $(K_\Ga, K_\Ga)$-bimodule
$K_\Ga$.
Since $G^t_{\La\Ga}$ is the functor defined by tensoring
with this bimodule,
the lemma follows.
\end{proof}

The next important theorem describes explicitly
the effect of a projective functor
on the projective indecomposable modules
$P(\la) := K_\La e_\la$ from \cite[$\S$5]{BS}.
Recall the graded vector space $R$ defined
immediately before (\ref{ntau}).

\begin{Theorem}
\label{pf}
Let $t$ be a proper $\La\Ga$-matching and $\gamma \in \Ga$.
\begin{itemize}
\item[\rm(i)]
We have that
$G^t_{\La\Ga} P(\ga)
\cong K_{\La\Ga}^t e_\ga \langle -\caps(t)\rangle$
as left $K_\La$-modules.
\item[\rm(ii)]
The module $G^t_{\La\Ga} P(\ga)$ is non-zero if and only if the rays of each
 upper line in $t \ga \overline{\ga}$
are oriented so that one is $\up$ and one is $\down$.
\item[\rm(iii)]
Assuming the condition from (ii) is satisfied,
define $\la \in \La$ by declaring that $\overline{\la}$ is
the upper reduction of $t \overline{\gamma}$,
and let $n$ be the number of upper circles
removed in the reduction process.
Then
$$
G_{\La \Ga}^t P(\ga) \cong
P(\la) \otimes R^{\otimes n} \langle
\cups(t)-\caps(t)\rangle.
$$
as graded left $K_\La$-modules (where the $K_\La$-action on
$P(\la) \otimes R^{\otimes n}$ comes from its action on the
first tensor factor).
\end{itemize}
\end{Theorem}

\begin{proof}
Note to start with that
\begin{align*}
G^t_{\La\Ga} P(\ga) &=
K_{\La \Ga}^t \langle -\caps(t)\rangle\otimes_{K_\Ga} P(\ga) \\
&=
K_{\La \Ga}^t \otimes_{K_\Ga} K_\Ga e_\ga \langle -\caps(t)\rangle
\cong
K_{\La \Ga}^t e_\ga \langle -\caps(t)\rangle,
\end{align*}
as asserted in (i).

We next establish the forward implication of (ii).
Suppose that the rays
of some upper line in the diagram $t \ga \overline{\ga}$
are oriented so that both are $\up$ or both are $\down$.
For any $\nu \in \Ga$ such that $\nu \overline{\ga}$
is oriented, the rays in $\nu \overline{\ga}$
are oriented in the same way as in $\ga \overline{\ga}$.
It follows that there are no weights $\mu \in \La, \nu \in \Ga$ such that
both $\nu \overline{\ga}$ and $\mu t \nu$ are oriented.
Since $K_{\La \Ga}^t e_\ga$ has a basis indexed
by oriented circle diagrams of the form
$a \mu t \nu \overline{\ga}$,
this implies that $K_{\La \Ga}^t e_\ga = \{0\}$.
Hence also $G^t_{\La\Ga} P(\ga) = \{0\}$ by (i), as required.

Now to complete the proof
we assume that the rays of every upper line in $t \ga \overline{\ga}$
are oriented so one is $\up$ and one is $\down$.
Enumerate the $n$ upper circles in the diagram $t \overline{\ga}$
in some fixed order.
Then the map
$$
f:K_{\La \Ga}^t e_\ga \rightarrow
K_\La e_\la \otimes R^{\otimes n},\quad
(a \mu t \nu \overline{\ga}) \mapsto
(a \mu \overline{\la}) \otimes x_1 \otimes\cdots\otimes x_n,
$$
where $x_i$ is $1$ or $x$ according to whether the $i$th upper
circle of $t \overline{\ga}$ is
anti-clockwise or clockwise in $a \mu t \nu \overline{\ga}$,
is an isomorphism of left $K_\La$-modules.
Moreover $f$ is homogeneous of degree
$\cups(t)$ by Lemma~\ref{cdeg}.
Since $P(\la) = K_\La e_\la$
and $G^t_{\La\Ga} P(\ga) \cong K^t_{\La\Ga}e_\ga \langle -\caps(t)\rangle$,
this proves (iii).
It also shows that
$G^t_{\La\Ga} P(\ga)$ is non-zero,
completing the proof of (ii) as well.
\end{proof}

\begin{Corollary} $K_{\La \Ga}^t$ is
projective both as a left
$K_\La$-module and as a right $K_\Ga$-module.
\end{Corollary}

\begin{proof}
As a left $K_\La$-module, we have that
$K_{\La \Ga}^t \cong \bigoplus_{\ga \in \Ga} K_{\La \Ga}^t e_\ga$.
Each summand is projective by Theorem~\ref{pf}(i),(iii),
hence $K_{\La \Ga}^t$ is a projective left $K_\La$-module.
Moreover, by the anti-multiplicativity of $*$, we know that
$K_{\La \Ga}^t$ is isomorphic as a right $K_\Ga$-module
to the right $K_\Ga$-module obtained by twisting the
left $K_\Ga$-module $K_{\Ga \La}^{t^*}$ with
$*$.
As we have already established,
$K_{\Ga \La}^{t^*}$ is a projective left $K_\Ga$-module, hence
$K_{\La \Ga}^t$ is a projective right $K_\Ga$-module too.
\end{proof}

\begin{Corollary}\label{fp}
Projective functors are exact and send finitely generated
projectives to finitely generated projectives.
\end{Corollary}

\phantomsubsection{Filtrations by cell modules}
Recall the definition of the
{\em cell modules} $V(\mu)$ for $K_\La$ from \cite[$\S$5]{BS}:
for each $\mu \in \La$, $V(\mu)$ is the graded vector space with
homogeneous basis
$\left\{
(c \mu | \:\big|\:
\text{for all oriented cup diagrams $c \mu$}\right\}$
and the action of  $(a \la b) \in K_\La$
is defined by
\begin{equation}\label{Actby}
(a \la b) (c \mu| :=
\left\{
\begin{array}{ll}
s_{a \la b}(\mu)  (a \mu|
&\text{if $b^* = c$ and $a \mu$ is oriented,}\\
0&\text{otherwise,}
\end{array}\right.
\end{equation}
where $s_{a \la b}(\mu) \in \{0,1\}$ is the scalar
from \cite[Theorem 4.4]{BS}.
The following theorem describes explicitly
the effect of a projective functor on a cell module,
showing in particular that
a cell module is mapped to a module with a
multiplicity-free filtration
by cell modules.

\begin{Theorem}\label{vf}
Let $t$ be a proper $\La \Ga$-matching and $\ga \in \Ga$.
\begin{itemize}
\item[(i)]
The $K_\La$-module
$G^t_{\La\Ga} V(\ga)$
has a filtration
$$
\{0\} = M(0) \subset M(1) \subset \cdots \subset M(n)
= G^t_{\La\Ga}V(\ga)
$$
such that
$M(i) / M(i-1) \cong V(\mu_i) \langle \deg(\mu_i t \ga)  - \caps(t)\rangle$
for each $i$.
Here $\mu_1,\dots,\mu_n$ denote the elements of the set
$\{\mu \in \La\:|\:\mu t \ga \text{ is oriented}\}$ ordered so that
$\mu_i > \mu_j$ implies $i < j$.
\item[(ii)] The module $G_{\La\Ga}^t V(\ga)$
is non-zero if and only if each cup in
$t\ga$ is oriented (some could be clockwise and some anti-clockwise).
\item[(iii)] Assuming the condition in (ii) is satisfied,
the module $G^t_{\La\Ga} V(\ga)$ is indecomposable with irreducible
head isomorphic to $L(\la) \langle \deg(\la t \ga) -\caps(t)\rangle$,
where $\la \in \La$ is the unique weight such that
$\overline{\la}$ is the upper reduction of $t \overline{\ga}$;
equivalently, $\la t \ga$ is oriented and all its caps are anti-clockwise.
\end{itemize}
\end{Theorem}

\begin{proof}
We first claim that the vectors
$$
\left\{(a \mu_i t \ga \overline{\ga}) \otimes (\underline{\ga} \ga |
\:\big|\: \text{for $i=1,\dots,n$
and all oriented cup diagrams $a \mu_i$}\right\}
$$
give a basis for $G_{\La\Ga}^t V(\ga)$.
By Theorem~\ref{t1}(ii), the vectors
$(a \mu t \nu \overline{\nu}) \otimes (\underline{\nu} \nu \overline{\gamma})$
give a basis for $G_{\La \Ga}^t P(\ga)$.
By \cite[Theorem~5.1]{BS}, $V(\ga)$ is the quotient of $P(\ga)$
by the subspace spanned by the vectors $(c \nu \overline{\ga})$
for $\nu > \ga$.
Using exactness of the functor $G_{\La \Ga}^t$,
it follows that
$G_{\La \Ga}^t V(\ga)$
is the quotient of
$G_{\La \Ga}^t P(\ga)$
by the subspace spanned by the vectors
$(a \mu t \la b) \otimes (c \nu \overline{\ga})$ for
$\nu > \ga$.
By Theorem~\ref{t1}(i), this subspace is spanned already by the basis vectors
$(a \mu t \nu \overline{\nu}) \otimes(\underline{\nu} \nu \overline{\ga})$
for $\nu > \ga$.
Hence $G_{\La \Ga}^t V(\ga)$ has a basis given by the images of the vectors
$(a \mu t \ga \overline{\ga}) \otimes (\underline{\ga} \ga \overline{\ga})$,
which is
equivalent to the claim.

Now we let $M(0) := \{0\}$ and for $i=1,\dots,n$ let
$M(i)$ be the subspace of $G_{\La\Ga}^t V(\ga)$
generated by $M(i-1)$ and the vectors
$$
\left\{(a \mu_i t \ga \overline{\ga}) \otimes (\underline{\ga} \ga |
\:\big|\:\text{for all oriented cup diagrams $a \mu_i$}\right\}.
$$
Thanks to the previous paragraph, this defines a filtration
of $G_{\La\Ga}^t V(\ga)$ with
$M(n)  = G_{\La\Ga}^t V(\ga)$.
We observe moreover that each $M(i)$ is a $K_\La$-submodule of
$G_{\La\Ga}^t V(\ga)$.
This follows by Corollary~\ref{ideal4} and Theorem~\ref{t1}(i).
The quotient $M(i) / M(i-1)$ has basis given by the images of the vectors
$$
\left\{(c \mu_i t \ga \overline{\ga}) \otimes (\underline{\ga} \ga |
\:\big|\:\text{for all oriented cup diagrams $c \mu_i$}\right\}.
$$
Moreover by Theorem~\ref{cell4} we have that
\begin{multline*}
(a \la b) (c \mu_i t \ga \overline{\ga}) \otimes (\underline{\ga} \ga|
\equiv\\
\begin{cases}
s_{a \la b}(\mu_i) (a \mu_i t \ga \overline{\ga}) \otimes (\underline{\ga} \ga|
&\text{if $b = c^*$ and $a \mu_i$ is oriented,}\\
0&\text{otherwise,}
\end{cases}
\end{multline*}
working modulo $M(i-1)$.
Recalling (\ref{Actby}), it follows that the map
$$
M(i) / M(i-1) \rightarrow V(\mu_i) \langle \deg(\mu_i t \ga) -\caps(t)\rangle,
\quad
(c \mu_i t \ga \overline{\ga}) \otimes (\underline{\ga} \ga|
\mapsto (c \mu_i|
$$
is an isomorphism of graded $K_\La$-modules.
This proves (i).

Now we deduce (ii). If there is a cup in the diagram $t \gamma$
that is not oriented,
then there are no $\mu \in \La$ such that $\mu t \ga$ is oriented.
Hence $G_{\La\Ga}^t V(\ga) = \{0\}$ by (i).
Conversely, if every cup in $t \gamma$ is oriented,
then the weight $\lambda$ defined in the statement of (iii) is a weight such that
$\lambda t \gamma$ is oriented, hence
$G_{\La\Ga}^t V(\ga)$ is non-zero.

Finally to prove (iii), we ignore the grading for a while.
The cell filtration from (ii) is multiplicity-free. Since each cell module
has irreducible head, it follows easily that
$G_{\La\Ga}^t V(\ga)$ has multiplicity-free head.
On the other hand, $G_{\La\Ga}^t V(\ga)$ is a quotient of
$G_{\La\Ga}^t P(\ga)$, which by
Theorem~\ref{pf}(iii) is a direct sum of copies of $P(\la)$.
Hence the head of $G_{\La\Ga}^t V(\ga)$
is a direct sum of copies of $L(\la)$.
These two facts together imply that the head consists of
just one copy of $L(\la)$ (shifted by some degree).
It just remains to determine the degree shift, which follows from (i).
\end{proof}

\phantomsubsection{Adjunctions}
Let $t$ be a proper $\La\Ga$-matching. We
are going
to prove that the functors
$G_{\Ga\La}^{t^*}$
and $G_{\La\Ga}^t$ form an adjoint pair (up to some
degree shifts).
Define a linear map
\begin{equation}\label{phidef}
\phi: K_{\Ga\La}^{t^*} \otimes K_{\La \Ga}^{t} \rightarrow K_\Ga
\end{equation}
as follows.
If $t$ is not a proper $\La\Ga$-matching we simply take $\phi:= 0$.
Now assume that $t$ is proper and
take basis vectors
$(a \la t^* \nu d) \in K_{\Ga\La}^{t^*}$ and
$(d' \kappa t \mu b) \in K_{\La \Ga}^{t}$.
Let $c$ be the upper reduction of $t^* d$.
If $d' = d^*$
 and all mirror image pairs of
upper and lower circles in
$t^* d$ and $d^* t$, respectively, are oriented in {\em opposite} ways
in the diagrams $a \la t^* \nu d$ and $d^* \kappa t \mu b$ then
we set
\begin{align}\label{theform}
\phi((a \la t^* \nu d) \otimes (d' \kappa t \mu b))
&:=
(a \la c) (c^* \mu b)\\\intertext{Otherwise, we set}
\label{theform2}
\phi((a \la t^* \nu d) \otimes (d' \kappa t \mu b)) &:= 0.
\end{align}

\begin{Lemma}\label{phic}
The map $\phi: K_{\Ga \La}^{t^*} \otimes K_{\La \Ga}^{t} \rightarrow K_\Ga$
is a homogeneous
$(K_\Ga, K_\Ga)$-bimodule homomorphism
of degree $-2\caps(t)$.
Moreover it is $K_\La$-balanced, so it factors through the
quotient to induce a map
$\bar\phi:K_{\Ga \La}^{t^*} \otimes_{K_\La} K_{\La \Ga}^{t}
\rightarrow K_\Ga$.
\end{Lemma}

\begin{proof}
We may as well assume that $t$ is a proper $\La\Ga$-matching, as
the lemma is trivial if it is not.
Now we proceed in several steps.

\vspace{1mm}
\noindent
{\em Step one: $\phi$
is homogeneous of degree $-2\caps(t)$.}
To see this, let notation be as in (\ref{theform}).
Suppose $p$ (resp.\ $q$) of the upper circles in
$t^* d$ are clockwise (resp.\ anti-clockwise)
in the diagram $a \la t^* \nu d$.
Then $q$ (resp.\ $p$) of the lower circles in
$d^* t$ are clockwise (resp.\ anti-clockwise).
By Lemma~\ref{cdeg} we have that
\begin{align*}
\deg(a \la t^* \nu d) &=
\deg(a \la c) +\cups(t^*) + p-q,\\
\deg(d^*\kappa t \mu b)&= \deg(c^* \mu b) +\caps(t) + q-p .
\end{align*}
Noting $\cups(t^*) = \caps(t)$, this shows that
$$
\deg((a \la c)(c^*\mu b))
=\deg((a \la t^* \nu d)\otimes (d^* \kappa t \mu b)) - 2\caps(t).
$$

\vspace{1mm}
\noindent
{\em Step two: the analogous statement at the level
of Khovanov algebras.}
Assume that $\La$ and $\Ga$ are Khovanov blocks.
Then we can also define a linear map
$$
\psi:H_{\Ga \La}^{t^*} \otimes H_{\La \Ga}^{t} \rightarrow H_\Ga
$$
by exactly the same formulae
(\ref{theform})--(\ref{theform2})
that were used to define $\phi$.
We are going to prove that $\psi$
is an $(H_\Ga, H_\Ga)$-bimodule homomorphism
and an $H_\La$-balanced map.
To see this, it is quite obvious from the definition that
$\psi$ is a left $H_\Ga$-module homomorphism, since
only the circles that do not meet the bottom number line
get changed in the passage from
$a \la t^* \nu d$ to $a \la c$, and
these play no role in the left $H_\Ga$-module structure.
Similarly, $\psi$ is a right $H_\Ga$-module homomorphism.
It remains to prove that $\psi$ is $H_\La$-balanced.
Introduce a new map
$$
\omega: H_{\Ga \La \Ga}^{t^* t} \rightarrow H_\Ga
$$
as follows. Take a basis vector
$(a \la t^* \mu t \nu b) \in H_{\Ga \La \Ga}^{t^* t}$.
If any of the internal circles in the diagram $t^* t$ are
anti-clockwise in $a \la t^* \mu t \nu b$,
we map it to zero. Otherwise,
let $u$ be the reduction of $t^* t$ and consider the diagram
$a \la u \nu b$. This has a symmetric middle section
containing $u$ so it makes sense to
iterate the surgery procedure from \cite[$\S$3]{BS} to this middle
section to obtain a linear combination of basis vectors
of $H_\Ga$.
We define the image of $(a \la t^* \mu t b)$ under the map $\omega$
to be this linear combination.
The main task now is to prove that
\begin{equation}\label{toprove}
\psi = \omega \circ m,
\end{equation}
where $m$ denotes the multiplication map
$H_{\Ga \La}^{t^*} \otimes H_{\La \Ga}^{t} \rightarrow H_{\Ga \La \Ga}^{t^*t}$
from  (\ref{oldmult}).
Since we know by (\ref{assoc})
that $m$ is $H_\La$-balanced,
(\ref{toprove}) immediately implies that $\psi$ is $H_\La$-balanced.

To prove (\ref{toprove}),
take basis vectors
$(a \la t^* \nu d) \in H_{\Ga\La}^{t^*}$ and
$(d' \kappa t \mu b) \in H_{\La \Ga}^t$. Since both maps $\psi$
and $m$ are zero on
$(a \la t^* \nu d) \otimes (d' \kappa t \mu b)$ if $d'\not= d^*$, we may
assume $d'=d^*$.
Now we proceed to show that
\begin{equation}\label{toprove2}
\psi((a \la t^* \nu d) \otimes (d^* \kappa t \mu b))
=
\omega((a \la t^* \nu d)  (d^* \kappa t \mu b))
\end{equation}
by induction on $\caps(t)$.
In the base case
$\caps(t) = 0$, there are no upper circles in $t^* d$
and no internal circles in the diagram $t^* t$.
Moreover, there
is a natural bijection between the caps in the diagram
$t^* d$ and the caps in its upper reduction $c$.
To compute
$\psi(
(a \la t^* \nu d) \otimes (d^* \kappa t \mu b))
= (a \la c) (c^* \mu b)$
involves applying the surgery procedure to eliminate each of these caps in
turn. On the other hand, to compute
$\omega((a \la t^* \nu d) (d^* \kappa t \mu b))$
involves first applying the surgery procedure to eliminate all of the
caps in $d$ (that is what the multiplication $m$ does)
then applying it some more to eliminate the remaining
caps in $t^*$ (that is what $\omega$ does).
The result is clearly the same either way.

So now assume that $\caps(t) > 0$ and that (\ref{toprove2})
has been proved for all smaller cases. Suppose first
that the diagram $t^* d$ contains a small circle (a circle with
just one cup and cap).
If this circle and its mirror image are oriented in the same
ways in $t^* \nu d$ and $d^* \kappa t$ then $\psi$ produces zero.
If they are both clockwise circles then
the product $(a \la t^* \nu d)(d^* \kappa t \mu b)$ is
 zero,
and if they are both anti-clockwise circles
then
there is an anti-clockwise internal circle
in the product $(a \la t^* \nu d)(d^* \kappa t \mu b)$ so $\omega$ produces
zero.
Hence we may assume the small circle and its mirror image are oriented
in opposite ways in $t^* \nu d$ and $d^*\kappa t$.
Now we can just remove these two circles (and
the vertices they pass through on the number lines)
to obtain simpler diagrams
$a \la^* t_1^*\nu_1 d_1$ and $d_1^* \kappa_1 t_1 \mu b$
with $\cups(t_1) < \cups(t)$.
Using the definitions it is easy to see that
\begin{align*}
\psi((a \la t^* \nu d) \otimes (d^* \kappa t \mu b))
&=
\psi((a \la t_1^* \nu_1 d_1) \otimes (d_1^* \kappa_1 t_1 \mu b)),\\
\omega((a \la t^* \nu d) (d^* \kappa t \mu b))
&=
\omega((a \la t_1^* \nu_1 d_1) (d_1^* \kappa_1 t_1 \mu b)).
\end{align*}
By induction the right hand sides of these equations are equal.
Hence the left hand sides are equal too.

We are now reduced to the situation that $\caps(t) > 0$
and $t^*d$ contains no small circles.
Then there must a kink on the top number
line of $t^*d$ with a
mirror image kink in $d^* t$.
Let $a \la^* t_1^* \nu_1 d_1$ and $d_1^* \kappa_1 t_1^* \mu b$
be the diagrams obtained by straightening these kinks (removing
two vertices from the number lines in the process):
$$
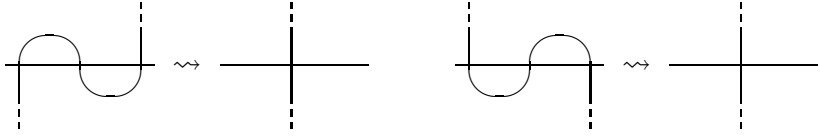
\begin{picture}(0,47)
\put(-153,21){\line(1,0){56}}
\put(-148,21){\line(0,-1){12}}
\put(-102,21){\line(0,1){12}}
\dashline{3}(-102,33)(-102,45)
\dashline{3}(-148,9)(-148,-3)
\put(-136.5,21){\oval(23,23)[t]}
\put(-113.5,21){\oval(23,23)[b]}
\put(-90,18.5){$\rightsquigarrow$}
\put(-72,21){\line(1,0){56}}
\put(-45,21){\line(0,1){12}}
\put(-45,21){\line(0,-1){12}}
\dashline{3}(-45,33)(-45,45)
\dashline{3}(-45,9)(-45,-3)

\put(17,21){\line(1,0){56}}
\put(22,21){\line(0,1){12}}
\put(68,21){\line(0,-1){12}}
\dashline{3}(22,33)(22,45)
\dashline{3}(68,9)(68,-3)
\put(33.5,21){\oval(23,23)[b]}
\put(56.5,21){\oval(23,23)[t]}
\put(80,18.5){$\rightsquigarrow$}
\put(125,21){\line(0,1){12}}
\put(125,21){\line(0,-1){12}}
\dashline{3}(125,33)(125,45)
\dashline{3}(125,9)(125,-3)
\put(98,21){\line(1,0){56}}
\end{picture}
$$
Again we have that $\caps(t_1) < \caps(t)$
so are done by the induction hypothesis as soon as we can show that
\begin{align*}
\psi((a \la t^* \nu d) \otimes (d^* \kappa t \mu b))
&=
\psi((a \la t_1^* \nu_1 d_1) \otimes (d_1^* \kappa_1 t_1 \mu b)),\\
\omega((a \la t^* \nu d) (d^* \kappa t \mu b))
&=
\omega((a \la t_1^* \nu_1 d_1)  (d_1^* \kappa_1 t_1 \mu b)).
\end{align*}
The first equality is clear as $t^*d$ and $t^*_1d_1$ have the same upper
reduction
$c$.
To get the second equality,
observe to compute the product $(a \la t^* \nu d)(d^* \kappa t \mu b)$
we can follow exactly the same sequence of surgery procedures as
to compute $(a \la t_1^* \nu_1 d_1)(d_1^* \kappa_1 t_1 \mu b)$ then add one
additional surgery procedure at the end involving the kinks
that were removed. This final surgery procedure
creates one extra internal circle:
$$
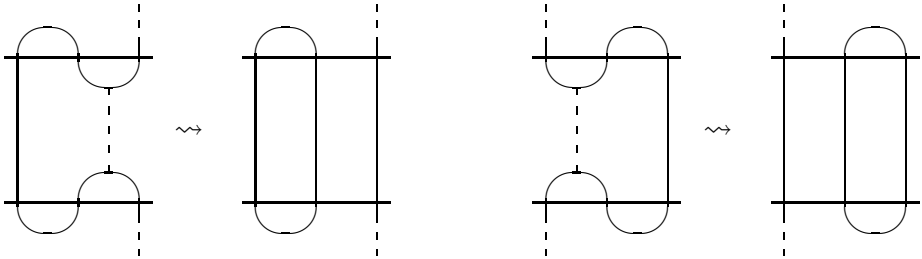
\begin{picture}(150,90)
\put(-99.5,15){\line(1,0){56}}
\put(-99.5,70){\line(1,0){56}}
\put(-94.5,70){\line(0,-1){55}}
\put(-83,15){\oval(23,23)[b]}
\put(-83,70){\oval(23,23)[t]}
\put(-60,70){\oval(23,23)[b]}
\put(-60,15){\oval(23,23)[t]}
\dashline{3}(-60,26.5)(-60,58.5)
\put(-48.5,15){\line(0,-1){5}}
\put(-48.5,75){\line(0,-1){5}}
\dashline{3}(-48.5,10)(-48.5,-5)
\dashline{3}(-48.5,90)(-48.5,75)
\put(-35,40){$\rightsquigarrow$}
\put(-9.5,15){\line(1,0){56}}
\put(-9.5,70){\line(1,0){56}}
\put(-4.5,70){\line(0,-1){55}}
\put(18.5,70){\line(0,-1){55}}
\put(41.5,70){\line(0,-1){55}}
\put(7,15){\oval(23,23)[b]}
\put(7,70){\oval(23,23)[t]}
\put(41.5,15){\line(0,-1){5}}
\put(41.5,75){\line(0,-1){5}}
\dashline{3}(41.5,10)(41.5,-5)
\dashline{3}(41.5,90)(41.5,75)

\put(100.5,15){\line(1,0){56}}
\put(100.5,70){\line(1,0){56}}
\put(151.5,70){\line(0,-1){55}}
\put(140,15){\oval(23,23)[b]}
\put(140,70){\oval(23,23)[t]}
\put(117,70){\oval(23,23)[b]}
\put(117,15){\oval(23,23)[t]}
\dashline{3}(117,26.5)(117,58.5)
\put(105.5,15){\line(0,-1){5}}
\put(105.5,75){\line(0,-1){5}}
\dashline{3}(105.5,10)(105.5,-5)
\dashline{3}(105.5,90)(105.5,75)
\put(165,40){$\rightsquigarrow$}
\put(190.5,15){\line(1,0){56}}
\put(190.5,70){\line(1,0){56}}
\put(195.5,70){\line(0,-1){55}}
\put(218.5,70){\line(0,-1){55}}
\put(241.5,70){\line(0,-1){55}}
\put(230,15){\oval(23,23)[b]}
\put(230,70){\oval(23,23)[t]}
\put(195.5,15){\line(0,-1){5}}
\put(195.5,75){\line(0,-1){5}}
\dashline{3}(195.5,10)(195.5,-5)
\dashline{3}(195.5,90)(195.5,75)
\end{picture}
$$
Since $\omega$ sends anti-clockwise internal circles to zero, we are only interested in terms in which
this extra internal circle is clockwise, so that the other circle
gets oriented at the end in the same way as it was prior to the last surgery procedure.
We conclude that the terms in the product
$(a \la t^* \nu d) (d^* \kappa t \mu b)$ in which the extra circle is clockwise
exactly match the terms in the product
$(a \la t_1^* \nu_1 d_1)  (d_1^* \kappa_1 t_1 \mu b)$.
So $\omega$ takes the same value on them both.

\vspace{1mm}
\noindent{\em Step three:
 $\phi$ is a $K_\La$-balanced $(K_\Ga,K_\Ga)$-bimodule
homomorphism.}
It remains to deduce this
statement about $\phi$
from the analogous statement about $\psi$ at the level of Khovanov algebras
just established in step two.
This is done by reducing first to the case that
$\Ga$ and $\La$ are blocks of bounded weights by a direct limit
argument involving (\ref{exxx}),
then passing from there to the Khovanov algebra setting by taking
closures using (\ref{clmap}). We omit the details.
(It is also possible to give a direct proof of this
statement by mimicing the arguments from step two in terms of
the generalised surgery procedure from \cite[$\S$6]{BS}.)
\end{proof}

\begin{Theorem}\label{gbi}
There is a graded $(K_{\La},K_\Ga)$-bimodule isomorphism
$$
\hat\phi:
K_{\La\Ga}^{t}\langle -2\caps(t)\rangle
\stackrel{\sim}{\rightarrow}
\hom_{K_\Ga}(K_{\Ga\La}^{t^*}, K_\Ga)
$$
sending $y \in K_{\La\Ga}^t$ to the
homomorphism
$\hat\phi(y): K_{\Ga\La}^{t^*} \rightarrow K_\Ga,
x \mapsto \phi(x \otimes y)$.
\end{Theorem}

\begin{proof}
We know already that $\phi$ is of degree $-2\caps(t)$ by
Lemma~\ref{phic},
hence the map
$\hat\phi$
is homogeneous of degree 0.
To check it is a left $K_\La$-module homomorphism,
take $u \in K_\La$, $y \in K_{\La\Ga}^t$ and $x \in K_{\Ga\La}^{t^*}$.
Using the fact that $\phi$ is $K_\La$-balanced,
we get that
$$
(u \hat \phi(y))(x) =
\hat \phi(y)(xu) = \phi(xu \otimes y)
= \phi(x \otimes uy)
= \hat\phi(uy)(x).
$$
Hence $u \hat \phi(y) = \hat \phi(uy)$ as required.
To check that it is a right $K_\Ga$-module homomorphism
suppose also that $v \in K_\Ga$.
Then
$$
(\hat\phi(y) v)(x)
= (\hat \phi(y)(x)) v
= \phi(x \otimes y) v
= \phi(x \otimes yv)
= \hat \phi(yv)(x).
$$
So $\hat \phi(y) v = \hat \phi(yv)$ as required.
It remains to check that $\hat\phi$ is a vector space isomorphism.
For this it is sufficient to show that the restriction
$$
\hat\phi: e_\la K_{\La \Ga}^t \rightarrow
e_\la \hom_{K_\Ga}(K_{\Ga\La}^{t^*}, K_\Ga)
= \hom_{K_\Ga}(K_{\Ga\La}^{t^*} e_\la, K_\Ga)
$$
is an isomorphism for each $\la \in \La$.
We may as well assume that $e_\la K_{\La \Ga}^t$ is non-zero,
since this is equivalent to $K_{\Ga\La}^{t^*} e_\la$ being non-zero.
Let $\ga \in \Ga$ be defined so that
$\overline{\ga}$ is the upper reduction of $t^* \overline{\la}$.
Let $n$ be the number of upper circles removed in the reduction process.
Then by Theorem~\ref{pf}(i),(iii) we have that
$K_{\Ga \La}^{t^*} e_\la \cong K_\Ga e_\ga \otimes R^{\otimes n}$
as a left $K_\Ga$-module (for the rest of the proof we are ignoring
the grading).
Similarly,
$e_\la K_{\La \Ga}^t \cong e_\ga K_\Ga \otimes R^{\otimes n}$
as right $K_\Ga$-modules.
Transporting the map $\hat \phi$ through these isomorphisms,
the thing we are trying to prove is equivalent to showing that the map
\begin{align*}
e_\ga K_\Ga \otimes R^{\otimes n}
&\rightarrow\hom_{K_\Ga}(K_\Ga e_\ga \otimes R^{\otimes n},
K_\Ga),\\
v \otimes x_1 \otimes \cdots \otimes x_n&\mapsto \big(u \otimes y_1 \otimes \cdots \otimes y_n
\mapsto\tau(x_1y_1) \cdots \tau(x_ny_n)uv\big)
\end{align*}
is an isomorphism,
where $\tau$
is the counit from (\ref{ntau}).
This quickly reduces to checking that the natural map
$e_\ga K_\Ga \rightarrow \hom_{K_\Ga}(K_\Ga e_\ga,
K_\Ga)$ defined by right multiplication is an isomorphism, which
is well known.
\end{proof}

\begin{Corollary}
\label{iso}
There is a canonical isomorphism
$$
\hom_{K_\Ga}\big(K_{\Ga\La}^{t^*}, ? \big)
\cong
K_{\La \Ga}^t
\langle-2\caps(t) \rangle
\otimes_{K_\Ga} ?
$$
of functors from $\mod{K_\Ga}$ to $\mod{K_\La}$.
\end{Corollary}

\begin{proof}
There is an obvious natural homomorphism
$$
\hom_{K_\Ga}(K_{\Ga\La}^{t^*}, M)
\cong
\hom_{K_\Ga}(K_{\Ga \La}^{t^*}, K_\Ga) \otimes_{K_\Ga} M
$$
for any $K_\Ga$-module $M$.
Since $K_{\Ga\La}^{t^*}$ is a projective left $K_\Ga$-module,
it is actually an isomorphism.
Now compose with the inverse of the
bimodule isomorphism from Theorem~\ref{gbi}.
\end{proof}

\begin{Corollary}\label{adj}
There is a canonical degree zero adjunction making
$$
(G^{t^*}_{\Ga\La}\langle \cups(t)-\caps(t)\rangle, G^t_{\La\Ga})
$$
into an adjoint pair of functors between
$\mod{K_\La}$ and $\mod{K_\Ga}$.
\end{Corollary}

\begin{proof}
Combine the standard adjunction
making
$(K_{\Ga\La}^{t^*} \otimes_{K_\La} ?, \hom_{K_\Ga}(K_{\Ga\La}^{t^*}, ?))$
into an adjoint pair
with Corollary~\ref{iso},
and recall the degree shift in (\ref{pfun}).
\end{proof}

\phantomsubsection{Projective functors commute with duality}
Next we show that projective functors
commute with the $\circledast$ duality
introduced just before \cite[(5.4)]{BS}.

\begin{Theorem}
\label{duality}
For a proper $\La \Ga$-matching $t$ and any
graded $K_\Ga$-module $M$, there is a natural isomorphism
$G^t_{\La \Ga}({M}^\circledast)
\cong
(G_{\La\Ga}^t M)^\circledast$
of graded $K_\La$-modules.
\end{Theorem}

\begin{proof}
In view of the grading shift in (\ref{pfun}), it suffices to construct
a natural $K_\La$-module isomorphism
$$
K_{\La \Ga}^t \otimes_{K_\Ga}({M}^\circledast)
\stackrel{\sim}{\rightarrow}
(K_{\La \Ga}^t \otimes_{K_\Ga} M)^\circledast
$$
that is homogeneous of degree $-2\caps(t)$.

We first define a linear map
$$
\theta:K_{\La\Ga}^t \otimes ({M}^\circledast) \otimes K_{\La\Ga}^t \otimes {M}
\rightarrow \C
$$
by sending $x \otimes f \otimes y \otimes m$ to
$f(\phi(x^* \otimes y) m)$, where $\phi$ is the map from
(\ref{phidef}) and $*$ is the degree preserving map from (\ref{star2}).
Since $\phi$ is of degree $-2\caps(t)$ according to Lemma~\ref{phic},
the map $\theta$ is of degree $-2\caps(t)$ as well.
For $u \in K_\Ga$, we have that
\begin{align*}
\theta(xu &\otimes f \otimes y \otimes m)
= f(\phi((xu)^* \otimes y)m)
= f(\phi(u^* x^* \otimes y)m)\\
&=
f(u^* \phi(x^* \otimes y) m)=
(uf)(\phi(x^* \otimes y) m)
=
\theta(x \otimes uf \otimes y \otimes m)
\end{align*}
and
\begin{align*}
\theta(x \otimes f \otimes yu \otimes m)
&= f(\phi(x^* \otimes yu) m)\\
&= f(\phi(x^* \otimes y) um)
= \theta(x \otimes f \otimes y \otimes um).
\end{align*}
This shows that $\theta$ factors through the quotients to induce
a homogeneous linear map
$\left(K_{\La\Ga}^t \otimes_{K_\Ga}
({M}^\circledast)\right) \otimes \left(K_{\La\Ga}^t \otimes_{K_\Ga} {M}\right)
\rightarrow \C$ of degree $-2\caps(t)$.

Hence there is a well-defined
homogeneous linear map of degree $-2\caps(t)$
$$
\tilde\theta:K_{\La\Ga}^t \otimes_{K_\Ga}({M}^\circledast)
\rightarrow
(K_{\La\Ga}^t \otimes_{K_\Ga} M)^\circledast
$$
which sends a generator
$x \otimes f \in
K_{\La\Ga}^t \otimes_{K_\Ga}({M}^\circledast)$ to the function
\begin{align*}
\tilde\theta(x \otimes f):
K_{\La\Ga}^t \otimes_{K_\Ga} M
&\rightarrow \C,\\
y \otimes m &\mapsto \theta(x \otimes f \otimes y \otimes m)  =
f(\phi(x^* \otimes y) m).
\end{align*}
We check moreover that $\tilde\theta$ is a $K_\La$-module homomorphism:
for $v \in K_\La$ we have that
\begin{align*}
(v \tilde\theta(x \otimes f))(y \otimes m)
&=
(\tilde \theta((x) \otimes f))(v^*y \otimes m)
=
f(\phi(x^* \otimes v^*y)m)\\
&=
f(\phi(x^* v^* \otimes y) m)
=
f(\phi((vx)^* \otimes y)m)\\
&=
\tilde\theta(vx \otimes f)(y\otimes m),
\end{align*}
i.e. $v \tilde \theta(x \otimes f) = \tilde\theta(vx \otimes f)$.

It remains to prove that $\tilde\theta$ is a vector space isomorphism.
For this, it suffices to show for each $\la \in \La$ that
the restriction
$$
\tilde\theta: e_\la K_{\La \Ga}^t \otimes_{K_\Ga}({M}^\circledast)
\rightarrow
e_\la (K_{\La \Ga}^t \otimes_{K_\Ga} M)^\circledast.
$$
is an isomorphism.
Of course we can identify
$e_\la (
K_{\La \Ga}^t \otimes_{K_\Ga} M)^\circledast$
with the graded dual
$(e_\la K_{\La \Ga}^t \otimes_{K_\Ga} M)^\circledast$
of the vector space
$e_\la K_{\La \Ga}^t \otimes_{K_\Ga} M$.
We may as well assume that $e_\la K_{\La\Ga}^t \neq \{0\}$.
Let $\ga \in \Ga$ be defined so that
$\overline{\ga}$ is the upper reduction of
$t^* \overline{\la}$.
Let $n$ be the number of upper circles removed in the reduction process.
Then by Theorem~\ref{pf}(i),(iii) we have that
$$
e_\la K_{\La \Ga}^t \cong e_\ga K_\Ga \otimes R^{\otimes n}
$$
as
right $K_\Ga$-modules (ignoring gradings).
Transporting $\tilde\theta$ through this isomorphism,
our problem reduces to showing that the map
$$
(e_\ga K_\Ga \otimes_{K_\Ga} ({M}^\circledast)) \otimes R^{\otimes n}
\rightarrow \big((e_\ga K_\Ga \otimes_{K_\Ga} M) \otimes R^{\otimes n}\big)^\circledast
$$
that sends $(u \otimes f) \otimes x_1 \otimes \cdots \otimes x_n$ to the function
$(v \otimes m) \otimes y_1 \otimes\cdots\otimes y_n
\mapsto \tau(x_1y_1)  \cdots \tau(x_n y_n) f(u^* v m)$
is an isomorphism,
where $\tau$ denotes the counit from (\ref{ntau}).
To see this,
identify $e_\ga K_\Ga \otimes_{K_\Ga} M = e_\ga M$
and
$e_\ga K_\Ga \otimes_{K_\Ga}({M}^\circledast) = e_\ga ({M}^\circledast) =
(e_\ga M)^\circledast$.
\end{proof}

\phantomsubsection{Projective functors on irreducible modules}
The next theorem describes the effect of
projective functors on the
irreducible $K_\Ga$-modules
$\{L(\ga)\:|\:\ga \in \Ga\}$; recall in Theorems~\ref{pf}
and \ref{vf}
we already explained what they do to projective and cell modules.

\begin{Theorem}\label{lf}
Let $t$ be a proper $\La\Ga$-matching and $\ga \in \Ga$.
\begin{itemize}
\item[(i)]
In the graded Grothendieck group $[\mod{K_\La}]$, we have that
$$
[G_{\La \Ga}^t L(\ga)]
= \sum_{\mu} (q+q^{-1})^{n_\mu}
 [L(\mu)],
$$
where
$n_\mu$ denotes the number of lower
circles in $\underline{\mu} t$ and the sum is over all $\mu \in \La$ such that
\begin{itemize}
\item[\rm(a)]
$\underline{\ga}$
is the lower reduction of $\underline{\mu} t$;
\item[\rm(b)]
the rays of each
lower line in $\underline{\mu} \mu t$ are oriented so that
exactly one is $\up$ and one is $\down$.
\end{itemize}
\item[(ii)]
The module $G_{\La\Ga}^t L(\ga)$ is non-zero
if and only if the cups of $t \ga$ are oriented so they are all
anti-clockwise.
\item[(iii)]
Assuming the condition in (ii) is satisfied, define $\la \in \La$ by declaring
that $\overline{\la}$ is the upper reduction of $t \overline{\ga}$;
equivalently,
$\la t \ga$ is oriented and all its caps and cups are anti-clockwise.
Then $G_{\La\Ga}^t L(\ga)$ is a self-dual
indecomposable module
with irreducible head isomorphic to $L(\la)\langle -\caps(t)\rangle$.
\end{itemize}
\end{Theorem}

\begin{proof}
We first prove (i). We need to show
for any $\mu \in \La$ that
$$
(\dagger) :=
\sum_{j \in \Z} q^j \dim \hom_{K_\La}(P(\mu), G_{\La\Ga}^t
L(\ga))_j
$$
is non-zero if and only if the conditions (a) and (b) are satisfied,
when it equals
$(q+q^{-1})^{n_\mu}$.
By Corollary~\ref{adj} we have that
$$
(\dagger) =
\sum_{j \in \Z} q^j \dim
\hom_{K_\Ga}(G_{\Ga\La}^{t^*} P(\mu)\langle \cups(t)-\caps(t) \rangle,
L(\ga))_j.
$$
By Theorem~\ref{pf} we know that
$G_{\Ga\La}^{t^*}P(\mu)$ is non-zero if and only if
(b) is satisfied,
in which case
it is isomorphic to $P(\beta) \otimes R^{\otimes{n_\mu}}\langle\caps(t)-\cups(t)\rangle$
where $\beta \in \Ga$ is such that
the lower reduction of $\underline{\mu} t$ equals $\underline{\beta}$.
Hence ($\dagger$) is non-zero if and only if both (a) and (b) are satisfied,
in which case
\begin{align*}
(\dagger)=
 \sum_{j \in \Z} q^j \dim
\hom_{K_\Ga}(P(\gamma) \otimes R^{\otimes{n_\mu}},
L(\ga))_j
=
(q+q^{-1})^{n_\mu},
\end{align*}
as claimed.

Now we deduce (ii) and (iii).
As $G^t_{\La\Ga} L(\ga)$ is a quotient
of $G^t_{\La\Ga} V(\ga)$,
we may assume by Theorem~\ref{vf}(ii)
that each cup of $t\ga$ is oriented. By
Theorem~\ref{vf}(iii),
$G^t_{\La\Ga} V(\ga)$ has irreducible
head
$L(\la) \langle \deg(\la t \ga)-\caps(t)\rangle$.
Hence
$G^t_{\La\Ga} L(\ga)$ is either zero
or else it too has irreducible head
$L(\la) \langle \deg(\la t \ga)-\caps(t)\rangle$.
But we know already by (i)
that $G^t_{\La\Ga}L(\ga)$ has a composition factor
isomorphic to $L(\la)$ (shifted in degree)
if and only if the weight $\mu := \la$
satisfies the conditions (a) and (b) from (i).
In view of the definition of $\la$,
these conditions are equivalent simply to the assertion
that each cup of $t \ga$ is anti-clockwise,
i.e. $\deg(\la t \ga) = 0$.
We have now proved (ii), and moreover we have
established the statement from (iii) about the head.
Finally the fact that $G^t_{\La\Ga} L(\ga)$ is self-dual
follow from Theorem~\ref{duality}.
\end{proof}

\begin{Corollary}\label{fl}
The functor $G_{\La \Ga}^t$
sends finite dimensional modules to finite dimensional
modules.
\end{Corollary}

\begin{proof}
In view of the theorem, it remains to observe for $\ga \in \Ga$
that there are only finitely many $\la \in \La$
such that $\underline{\ga}$ is the lower reduction of $\underline{\la} t$.
This follows ultimately because there are only finitely many caps in $t$.
\end{proof}

\phantomsubsection{Indecomposable projective functors}
The final theorem of the section proves that
the projective functor from (\ref{pfun}) is
an indecomposable functor, i.e. it cannot be written as the
direct sum of two non-zero subfunctors.
The proof follows the same strategy as the proof of
\cite[Theorem 5.1]{Std}.
We need one preliminary lemma.
To formulate this, let $t$ be a proper $\La\Ga$-matching.
Define
\begin{equation}\label{sim}
\Ga(t) := \{\ga \in \Ga\:|\:K_{\La\Ga}^t\otimes_{K_\Ga} L(\ga) \neq \{0\}\}.
\end{equation}
Note Theorem~\ref{lf}(ii) gives an explicit combinatorial description of
$\Ga(t)$: it is the set of all weights $\ga \in \Ga$ such that
all cups in the composite diagram $t \ga$ are anti-clockwise cups.
Let $\stackrel{t}{\sim}$
denote the equivalence relation on $\Ga(t)$ generated by the property
that $\la \stackrel{t}{\sim} \mu$ if the composition multiplicity
$[V(\la):L(\mu)]$ is non-zero, for $\la, \mu \in \Ga(t)$.

\begin{Lemma}\label{block}
The set $\Ga(t)$ is a single $\stackrel{t}{\sim}$-equivalence class.
\end{Lemma}

\begin{proof}
Let $I$ index the vertices on the top number line of $t$ that are
at  the ends of line segments.
Take any pair of indices $i < j$ from $I$ that are neighbours
in the sense that no integer
between $i$ and $j$ belongs to $I$.
To prove the lemma it suffices to show that
$\la \sim \mu$ for any pair of weights
$\la, \mu \in \Ga(t)$ such that
\begin{itemize}
\item the $i$th and $j$th
vertices of $\la$ and $\mu$
are labelled $\up \down$ and $\down \up$, respectively;
\item all other vertices of $\la$ and $\mu$ are labelled in the same way.
\end{itemize}
For this, note from the
explicit combinatorial description of weights in $\Ga(t)$
that $\underline{\mu}$ has a cup joining vertices $i$ and $j$.
Hence $\underline{\mu}\la$ is an oriented cup diagram.
It follows that $[V(\la):L(\mu)] \neq 0$ by \cite[Theorem 5.2]{BS},
so $\la \stackrel{t}{\sim} \mu$ as required.
\end{proof}

\begin{Theorem}
The functor
$G^t_{\La\Ga}:\mod{K_\Ga} \rightarrow \mod{K_\La}$ is indecomposable.
\end{Theorem}

\begin{proof}
Suppose that $G^t_{\La\Ga} = F_1 \oplus F_2$ as a direct sum of
functors. Each $F_i$ is automatically exact.
To prove the theorem it suffices to show either that $F_1 = 0$
or that $F_2 = 0$. By exactness, this follows if we can show
either that $\Ga_1 = \varnothing$  or that $\Ga_2 = \varnothing$,
where
$\Ga_i := \{\ga \in \Ga\:|\:F_i L(\ga) \neq \{0\}\}$.

Take $\ga \in \Ga$. If $\ga \notin \Ga(t)$ then
$G^t_{\La\Ga} L(\ga) = \{0\}$, hence
both $F_1 L(\ga)$ and $F_2 L(\ga)$ are zero.
If $\ga \in \Ga(t)$ then $G^t_{\La\Ga}L(\ga)
= F_1 L(\ga) \oplus F_2 L(\ga)$ is a non-zero indecomposable module
by Theorem~\ref{lf}(iii). Hence exactly one of $F_1 L(\ga)$ or
$F_2 L(\ga)$ is non-zero. This shows that we have a partition
$$
\Ga(t) = \Ga_1 \sqcup \Ga_2.
$$
Now we claim for
weights $\la, \mu \in \Ga(t)$ satisfying $[V(\la):L(\mu)] \neq 0$ that
$\la$ and $\mu$ belong to the same one of the sets $\Ga_1$ or $\Ga_2$.
To see this, assume without loss of generality that $\la \in \Ga_1$.
As $F_1 L(\la) \neq \{0\}$ we deduce that $F_1 V(\la) \neq \{0\}$.
As $G^t_{\La\Ga} V(\la) =
F_1 V(\la) \oplus F_2 V(\la)$ is indecomposable by Theorem~\ref{vf}(iii),
we deduce that $F_2 V(\la) = \{0\}$.
As $[V(\la):L(\mu)] \neq 0$ we deduce that $F_2 L(\mu) = \{0\}$ too.
Hence $\mu \notin \Ga_2$, so $\mu \in \Ga_1$ as required.

Recalling the definition of the equivalence relation $\stackrel{t}{\sim}$
on $\Ga(t)$, the previous paragraph shows that $\Ga_1$ and $\Ga_2$
are both unions of $\stackrel{t}{\sim}$-equivalence classes.
In view of Lemma~\ref{block} this means either $\Ga_1 = \varnothing$ or
$\Ga_2 = \varnothing$.
\end{proof}

\section{Kazhdan-Lusztig polynomials and Koszulity}\label{skoszul}

Now we have enough basic tools in place to be able to prove
that the algebras $K_\La$ are Koszul.
The key step is the explicit construction of a linear projective
resolution of each cell module.

\phantomsubsection{Combinatorial definition of Kazhdan-Lusztig polynomials}
In this subsection we recall a beautiful closed formula for the
Kazhdan-Lusztig polynomials associated to
Grassmannians discovered by Lascoux and Sch\"utzenberger \cite{LS};
see also \cite{Z}.
Actually we are going reformulate \cite{LS} in terms of
cap diagrams to match the combinatorial language we are using everywhere
else; one advantage of this is that the reformulation also makes sense
for unbounded weights.

We begin by introducing some length functions.
For a block $\La$, let $I(\La)$ denote the (possibly infinite)
set of integers indexing vertices of
weights in $\La$ whose labels are {\em not} $\circ$ or $\times$.
For $\la, \mu \in \La$,
let $I(\la,\mu)$ denote the (necessarily finite) subset of $I(\La)$
indexing the vertices that are labelled differently in $\la$ and $\mu$.
For any $i \in I(\La)$, we define
\begin{multline}
\ell_i(\la,\mu)
:=
\#\{j \in I(\la,\mu)\:|\:j \leq i
\text{ and vertex $j$ of $\la$ is labelled $\down$}\}
\\
\qquad-
\#\{j \in I(\la,\mu)\:|\:j \leq i
\text{ and vertex $j$ of $\mu$ is labelled $\down$}\}.
\end{multline}
So $\ell_i(\la,\mu)$ counts how many more $\down$'s
there are in $\la$ compared to $\mu$ on the vertices to the left or equal
to the $i$th vertex.
We note that $\la \leq \mu$ in the Bruhat ordering if and only if
$\ell_i(\la,\mu) \geq 0$ for all $i \in I(\La)$.
Introduce the
{\em relative length function} on $\La$ by setting
\begin{equation}
\ell(\la,\mu) := \sum_{i \in I(\La)} \ell_i(\la,\mu).
\end{equation}
If $\la \leq \mu$, this is just
the minimum number of transpositions of neighbouring
$\down \up$ pairs needed to get from $\la$ to $\mu$, where
``neighbouring'' means separated only by $\circ$'s and $\times$'s.

Any cap diagram $c$ cuts the upper half space above the number line
into various open connected regions. We refer to these as the {\em chambers}
of the cap diagram $c$.
A {\em labelled cap diagram} $C$
is a cap diagram whose chambers have been labelled
by integers in such a way that
\begin{itemize}
\item external (unbounded) chambers are labelled by zero;
\item given two chambers separated by a cap,
the label in the inside chamber is greater than or equal to the label
in the outside chamber.
\end{itemize}
If $C$ is a labelled cap diagram we let $|C|$ denote the sum of its labels.
In any cap diagram, we refer to a cap in the diagram that does not
contain any smaller nested caps as a {\em small cap}.

Now we are ready to define
the (combinatorial) Kazhdan-Lusztig polynomials $p_{\la,\mu}(q)$
for each $\la,\mu \in \La$.
If $\la \not\leq \mu$ then we set $p_{\la,\mu}(q) := 0$.
If $\la \leq \mu$, let $D(\la,\mu)$ denote the set of all labelled cap diagrams
obtained by labelling the chambers of
$\overline{\mu}$ in such a way that
the label inside every small cap is
less than or equal to $\ell_i(\la,\mu)$, where $i$
indexes the leftmost vertex of the small cap.
Set
\begin{equation}\label{newp}
p_{\la,\mu}(q) := q^{\ell(\la,\mu)}\sum_{C \in D(\la,\mu)} q^{-2|C|}.
\end{equation}
It is not immediately obvious from (\ref{newp})
that $p_{\la,\mu}(q)$ is actually a polynomial in $q$, but that is clear
from Lemma~\ref{recurse} below.
We will write $p_{\la,\mu}^{(n)}$ for the
$q^n$-coefficient of $p_{\la,\mu}(q)$, so
\begin{equation}
p_{\la,\mu}(q) = \sum_{n \geq 0} p_{\la,\mu}^{(n)} q^n.
\end{equation}
We also define the corresponding matrix
\begin{equation}\label{pmat}
P_\La(q) := (p_{\la,\mu}(q))_{\la,\mu \in \La}.
\end{equation}
Here is an example illustrating the computation
of $p_{\la,\mu}(q)$ in practise:
$$
\begin{picture}(120,95)
\put(-121.5,6){$p_{\la,\mu}(q) =\:\:\:q^{16}(1+q^{-2})(1+q^{-2}+q^{-4})$.}
\put(-95.5,21){$\la = $}
\put(-64.8,23){\line(1,0){299}}
\put(-67.7,23.1){$\scriptstyle\down$}
\put(-44.7,23.1){$\scriptstyle\down$}
\put(-21.7,23.1){$\scriptstyle\down$}
\put(1.3,23.1){$\scriptstyle\down$}
\put(24.3,23.1){$\scriptstyle\down$}
\put(47.3,18.4){$\scriptstyle\up$}
\put(70.3,23.1){$\scriptstyle\down$}
\put(93.3,18.4){$\scriptstyle\up$}
\put(116.3,18.4){$\scriptstyle\up$}
\put(139.3,23.1){$\scriptstyle\down$}
\put(162.3,23.1){$\scriptstyle\down$}
\put(185.3,18.4){$\scriptstyle\up$}
\put(208.3,18.4){$\scriptstyle\up$}
\put(231.3,18.4){$\scriptstyle\up$}

\put(-123.4,33.6){$\ell_i(\la,\mu) = $}
\put(-66.7,33.6){$\scriptstyle 0$}
\put(-43.7,33.6){$\scriptstyle 0$}
\put(-20.7,33.6){$\scriptstyle 1$}
\put(2.3,33.6){$\scriptstyle 1$}
\put(25.3,33.6){$\scriptstyle 1$}
\put(48.3,33.6){$\scriptstyle 1$}
\put(71.3,33.6){$\scriptstyle 2$}
\put(94.3,33.6){$\scriptstyle 2$}
\put(117.3,33.6){$\scriptstyle 2$}
\put(140.3,33.6){$\scriptstyle 2$}
\put(163.3,33.6){$\scriptstyle 2$}
\put(186.3,33.6){$\scriptstyle 1$}
\put(209.3,33.6){$\scriptstyle 1$}
\put(232.3,33.6){$\scriptstyle 0$}

\put(-95.5,46){$\mu = $}
\put(-64.8,48){\line(1,0){299}}
\put(-67.6,48.1){$\scriptstyle\down$}
\put(-44.6,48.1){$\scriptstyle\down$}
\put(-21.6,43.4){$\scriptstyle\up$}
\put(1.4,48.1){$\scriptstyle\down$}
\put(24.3,48.1){$\scriptstyle\down$}
\put(47.3,43.4){$\scriptstyle\up$}
\put(70.3,43.4){$\scriptstyle\up$}
\put(93.3,43.4){$\scriptstyle\up$}
\put(116.3,43.4){$\scriptstyle\up$}
\put(139.3,48.1){$\scriptstyle\down$}
\put(162.3,48.1){$\scriptstyle\down$}
\put(185.3,48.1){$\scriptstyle\down$}
\put(208.3,43.4){$\scriptstyle\up$}
\put(231.3,48.1){$\scriptstyle\down$}

\put(119.1,48){\line(0,1){43}}
\put(142.1,48){\line(0,1){43}}
\put(165.1,48){\line(0,1){43}}
\put(234.1,48){\line(0,1){43}}
\put(-30.3,48){\oval(23,23)[t]}
\put(38.7,48){\oval(23,23)[t]}
\put(38.7,48){\oval(69,40)[t]}
\put(199.7,48){\oval(23,23)[t]}
\put(15.7,48){\oval(161,60)[t]}

\put(194.1,51.7){$\scriptstyle 0 / 1$}
\put(33,51.7){$\scriptstyle 0/1$}
\put(197.6,73.7){$\scriptstyle 0$}
\put(33,61.2){$\scriptstyle 0/1$}
\put(-32.7,66.7){$\scriptstyle 0$}
\put(-32.7,82.7){$\scriptstyle 0$}
\put(-32.7,51.7){$\scriptstyle 0$}
\put(127.7,73.7){$\scriptstyle 0$}
\put(150.7,73.7){$\scriptstyle 0$}

\end{picture}
$$

\begin{Remark}\label{lls}\rm
If $|\La| < \infty$, the polynomials $p_{\la,\mu}(q)$
are almost exactly
the same as the polynomials $Q^v_w(q)$ defined
in \cite[$\S$6]{LS}, which are shown in \cite[Th\'eor\`eme 7.8]{LS}
to be equal to the (geometric) Kazhdan-Lusztig polynomials associated
to Grassmannians in the sense of \cite{KL}.
To be precise,
for $\la \in \La$, let $w(\la)$ denote the word
in symbols $\alpha$ and $\beta$
obtained by reading the vertices of $\la$ from left to right
and writing $\alpha$ for each $\down$ and $\beta$ for each $\up$.
Then we have that
$$
p_{\la,\mu}(q)
= q^{\ell(\la,\mu)} Q^{w(\la)}_{w(\mu)}(q^{-2}).
$$
This equality follows by translating
the combinatorial description from \cite{LS}.
\end{Remark}

In the next subsection we need instead a recursive
description of the polynomials $p_{\la,\mu}(q)$
from \cite{LS}.
To formulate this,
given neighbouring indices $i < j$ from $I(\La)$, we
let
\begin{equation}\label{laij}
\La_{i,j}^{\!\vee \wedge}
:= \left\{\nu\in \La\:\Bigg|\:
\begin{array}{l}
\text{the $i$th vertex of $\nu$ is labelled $\down$}\\
\text{the $j$th vertex of $\nu$ is labelled $\up$}
\end{array}
\right\}.
\end{equation}
The following lemma repeats \cite[Lemme 6.6]{LS} (and its proof).

\begin{Lemma}\label{recurse}
The polynomials $p_{\la,\mu}(q)$ are
determined uniquely by the following properties.
\begin{itemize}
\item[(i)]
If $\la = \mu$ then $p_{\la,\mu}(q) = 1$
and if $\la \not \leq \mu$ then $p_{\la,\mu}(q) = 0$.
\item[(ii)]
If $\la < \mu$, pick neighbouring indices $i < j$ from $I(\La)$
such that $\la \in \La_{i,j}^{\!\vee\wedge}$.
Then
\begin{equation*}
p_{\la,\mu}(q) =
\left\{
\begin{array}{ll}
p_{\la',\mu'}(q) + q p_{\la'',\mu}(q)&\text{if $\mu \in \La_{i,j}^{\!\vee\wedge}$,}\\
q p_{\la'',\mu}(q)&\text{otherwise,}
\end{array}\right.
\end{equation*}
where $\la'$ and $\mu'$ denote the weights in some ``smaller'' block
obtained from $\la$
and $\mu$ by deleting vertices $i$ and $j$,
and $\la''$ is the weight obtained from $\la$ by transposing the labels
on vertices $i$ and $j$.
\end{itemize}
\end{Lemma}

\begin{proof}
We just check (ii). If $\mu \in \La_{i,j}^{\!\down\up}$
then $\ell_i(\la,\mu) = \ell_i(\la'',\mu)+1$ and
$\ell(\la,\mu) = \ell(\la',\mu') =
\ell(\la'',\mu)+1$.
There is a bijection
from $D(\la'',\mu)\sqcup D(\la',\mu')$
to $D(\la,\mu)$
mapping $C \in D(\la'',\mu)$ to itself
and $C \in D(\la',\mu')$ to the labelled cap diagram
obtained by adding vertices $i$ and $j$ and
a small cap between them labelled by $\ell_i(\la,\mu)$.
Using these observations
and (\ref{newp}), we get that $p_{\la,\mu}(q) =
p_{\la',\mu'}(q) + q p_{\la'',\mu}(q)$.
Instead if $\mu \notin \La_{i,j}^{\!\down\up}$ then
$\ell_i(\la,\mu) = \ell_i(\la'',\mu)+1$
and $D(\la,\mu) = D(\la'',\mu)$.
Again we get that $p_{\la,\mu}(q) = q p_{\la'',\mu}(q)$.
\end{proof}

\phantomsubsection{Linear projective resolutions of cell modules}
Recall the cell modules from (\ref{Actby}).
The following theorem is the central result of the section;
the proof here is based on \cite[Lemma 4.49]{B}.

\begin{Theorem}\label{pres}
For $\la \in \La$, there is an exact sequence
$$
\cdots
\stackrel{d_1}{\longrightarrow}
\cP_{1}(\la)
\stackrel{d_0}{\longrightarrow} \cP_{0}(\la)
\stackrel{\eps}{\longrightarrow} V(\la) \longrightarrow 0
$$
where $\cP_0(\la) := P(\la)$ and
$\cP_n(\la) := \bigoplus_{\mu \in \La}
p_{\la,\mu}^{(n)} P(\mu)\langle n \rangle$
for $n \geq 0$.
\end{Theorem}

\begin{proof}
As $\cP_0(\la)$ is the projective cover $P(\la)$ of $V(\la)$,
there is a surjection $\eps:\cP_0(\la) \rightarrow V(\la)$.
We show by simultaneous induction on $\defect(\la)$
and $n =0,1,\dots$
that there exist maps
$d_n:\cP_{n+1}(\la) \rightarrow \cP_{n}(\la)$
such that $\im\, d_n = \ker d_{n-1}$, interpreting
$d_{-1}$ as the map $\eps$.
If $\defect(\la) = 0$
then $\la$ is maximal in the Bruhat ordering, so
$P(\la) = V(\la)$ by \cite[Theorem~5.1]{BS}
and $\cP_1(\la)=\cP_2(\la)=\cdots=0$.
Hence all the maps $d_0,d_1,\dots$ are necessarily zero and $\eps$ is an isomorphism.
This verifies the assertion for $\defect(\la) = 0$.

Now we assume that $n \geq 0$, $\defect(\la) > 0$ and that we have already
constructed the maps $d_m:\cP_{m+1}(\la) \rightarrow \cP_m(\la)$
for all $0 \leq m < n$.
Choose neighbouring indices
$i < j$ from $I(\La)$
such that $\la\in\La_{i,j}^{\!\vee \wedge}$.
Recalling (\ref{laij}), let $\Ga$ denote the ``smaller'' block
consisting of all weights
obtained by removing vertices $i$ and $j$ from
the weights in $\La^{\!\vee\wedge}_{i,j}$, and let
$
\La_{i,j}^{\!\vee\wedge} \rightarrow \Ga,\:
\nu \mapsto \nu'$
be the obvious bijection defined by letting
$\nu'$ denote the weight
obtained from $\nu$ by deleting vertices $i$ and $j$.
Let $t$ be the proper $\La\Ga$-matching with a cap joining vertices $i$ and $j$
on the bottom number line and vertical line segments joining all other
corresponding pairs of
vertices in the bottom and top number lines (skipping vertices labelled
$\circ$ or $\times$).
The definition of $t$ means that $\overline{\mu}$ is the upper reduction of
$t \overline{\mu'}$ for every $\mu \in \La^{\!\vee\wedge}_{i,j}$, hence
\begin{equation}\label{step1}
G^t_{\La\Ga} P(\mu')  \cong P(\mu)
\end{equation}
by Theorem~\ref{pf}.
Also by Theorem~\ref{vf}(i), there is a short exact sequence
\begin{equation}\label{step2}
0 \longrightarrow V(\la'')
\stackrel{f}{\longrightarrow}
G^t_{\La\Ga} V(\la')
\longrightarrow V(\la)\langle -1\rangle
\longrightarrow 0.
\end{equation}
The idea now is to apply the induction hypothesis to the left and middle term of this sequence and then deduce the result for the right hand side by taking the cone of $f$.

By the induction hypothesis, there are  exact sequences
\begin{align*}
\cP_{n}(\la'')
{\rightarrow} \cdots
{\rightarrow} \cP_0(\la'')
{\rightarrow} V(\la'') \rightarrow 0&,\\
\cP_{n+1}(\la')
{\rightarrow}
\cP_{n}(\la')
{\rightarrow} \cdots
{\rightarrow} \cP_0(\la') {\rightarrow} V(\la')\rightarrow 0&
\end{align*}
with $\cP_n(\la'') = \bigoplus_{\mu \in \La} p_{\la'',\mu}^{(n)}
P(\mu)\langle n \rangle$ and $\cP_{n+1}(\la')
= \bigoplus_{\mu \in \La^{\!\vee\wedge}_{i,j}} p_{\la',\mu'}^{(n+1)} P(\mu')
\langle n +1\rangle$.
Applying the exact functor $G^t_{\La\Ga}$ to the second of
these and then
using \cite[2.2.6]{Wei}
to lift the map $f$ from (\ref{step2}) to a chain map, we obtain
a commuting diagram
$$
\begin{CD}
&&\cP_{n}(\la'')
&\rightarrow\cdots\rightarrow&\cP_0(\la'')&\rightarrow&V(\la'')&\rightarrow 0\\
&&@VVV@VVV@VfVV\\
G^t_{\La\Ga}\cP_{n+1}(\la')&\rightarrow&
G^t_{\La\Ga}\cP_{n}(\la')&\rightarrow\cdots\rightarrow&G^t_{\La\Ga}\cP_0(\la')&\rightarrow&G^t_{\La\Ga}
V(\la')&\rightarrow 0
\end{CD}
$$
with exact rows.
The total complex of this double complex is exact by
\cite[2.7.3]{Wei}, so gives an exact sequence
\begin{multline*}
\cP_{n}(\la'')\oplus G^t_{\La\Ga} \cP_{n+1}(\la')
\rightarrow\cdots\rightarrow
 V(\la'') \oplus G^t_{\La\Ga} \cP_0(\la')
\rightarrow
G^t_{\La\Ga} V(\la')\rightarrow 0.
\end{multline*}
To this sequence, there is an obvious injective chain map
from the exact
sequence $
0 \rightarrow \cdots \rightarrow 0 \rightarrow
V(\la'') \rightarrow V(\la'')
\rightarrow 0$.
Taking the quotient using (\ref{step2}) and \cite[Exercise 1.3.1]{Wei}, we
obtain another exact sequence
$$
\cP_{n}(\la'') \oplus G^t_{\La\Ga} \cP_{n+1}(\la')
\rightarrow\cdots\rightarrow
G^t_{\La\Ga} \cP_0(\la')
\rightarrow
V(\la)\langle -1 \rangle\rightarrow 0.
$$
Recalling Lemma~\ref{recurse} and (\ref{step1}), we
have that $$
G^t_{\La\Ga} \cP_0(\la') = G^t_{\La\Ga} P(\la') \cong P(\la)\langle -1 \rangle = \cP_0(\la)\langle -1 \rangle
$$ and
\begin{align*}
\cP_{n}(\la'')
\oplus G^t_{\La\Ga} \cP_{n+1}(\la')
&\cong
\bigoplus_{\mu \in \La}
p_{\la'',\mu}^{(n)}
P(\mu) \langle n \rangle
\oplus \bigoplus_{\mu \in \La_{i,j}^{\!\vee \wedge}}
p_{\la',\mu'}^{(n+1)} P(\mu)\langle n \rangle\\
&\cong
\bigoplus_{\mu \in \La} p_{\la,\mu}^{(n+1)} P(\mu) \langle n
\rangle =
\cP_{n+1}(\la)\langle -1\rangle.
\end{align*}
So, shifting degrees by one, our exact sequence can be
rewritten as an exact sequence
$$
\cP_{n+1}(\la) \stackrel{d_n}{\longrightarrow}\cdots
\stackrel{d_0}{\longrightarrow} \cP_0(\la) \stackrel{\eps}{\longrightarrow} V(\la)
\longrightarrow 0.
$$
This constructs a map $d_n$
such that $\im \,d_n = \ker d_{n-1}$,
where for $n > 0$ the map $d_{n-1}$ is as constructed
in the previous iteration of the induction.
\end{proof}

\begin{Corollary}\label{invmat}
The matrix $P_\La(-q)$ is the inverse
of the $q$-decomposition
matrix $D_\La(q)$
from \cite[(5.13)]{BS}.
\end{Corollary}

\begin{proof}
Recalling \cite[(5.14)]{BS}
and the Cartan matrix
$C_\La(q) = (c_{\la,\mu}(q))_{\la,\mu \in \La}$ from
\cite[(5.6)]{BS}, we have by Theorem~\ref{pres} that
\begin{align*}
d_{\la,\mu}(q)
 &= \sum_{j \geq 0} q^j
\dim \hom_{K_\La}(P(\la),V(\mu))_j
\phantom{hehrtheqrgqergqerd|h}\\
&=
\sum_{i,j \geq 0} (-1)^iq^j \dim \hom_{K_\La}(P(\la), \cP_i(\mu))_j\\
&=
\sum_{\nu \in \La}
\sum_{i,j \geq 0} (-1)^iq^j p_{\mu,\nu}^{(i)}
\dim \hom_{K_\La}(P(\la), P(\nu) \langle i \rangle)_j
\end{align*}\begin{align*}
\phantom{d_{\la,\mu}(q)}&=
\sum_{\nu \in \La}
\Big(\sum_{i\geq 0} (-q)^i  p_{\mu,\nu}^{(i)}\Big)\Big(\sum_{j \geq 0}q^j
\dim \hom_{K_\La}(P(\la), P(\nu))_j\Big)\\
&=
\sum_{\nu \in \La}
p_{\mu,\nu}(-q)c_{\la,\nu}(q).
\end{align*}
Using also the familiar factorisation \cite[(5.17)]{BS}, this shows that
$$
D_\La(q) = C_\La(q) P_\La(-q)^T =
D_\La(q) D_\La(q)^T P_\La(-q)^T.
$$
Now multiply on the left by the inverse matrix
$D_\La(q)^{-1}$ (which exists in all cases as $D_\La(q)$ is a unitriangular matrix)
and transpose.
\end{proof}

The next corollary identifies the polynomials $p_{\la,\mu}(q)$
with the (representation theoretic)
{Kazhdan-Lusztig polynomials} associated to the category
$\mod{K_\La}$
in the sense of
Vogan \cite{V}.

\begin{Corollary}\label{indep}
For $\la,\mu \in \La$, we have that
$$
p_{\la,\mu}(q) =
\sum_{i \geq 0} q^i \dim \ext^i_{K_\La}(V(\la), L(\mu)).
$$
Moreover,
$\ext^{i}_{K_\La}(V(\la),L(\mu))_{-j} = 0$
unless $i = j \equiv \ell(\la,\mu) \pmod{2}$.
\end{Corollary}

\begin{proof}
Apply the functor $\hom_{K_\La}(?, L(\mu))$ to the projective resolution of
$V(\la)$ constructed in Theorem~\ref{pres} to obtain
a cochain complex
$0 \rightarrow \cC^0 \rightarrow \cC^1 \rightarrow \cdots$
with
$$
\cC^i = \bigoplus_{j \in \Z} \cC^{i,j} =
\bigoplus_{j \in \Z}
\hom_{K_\La}(\cP_i(\la), L(\mu))_{-j},
$$
the cohomology of which computes $\ext_{K_\La}^i(V(\la),L(\mu))$.
We get from the definition of $\cP_i(\la)$ that
$\dim\cC^{i,j} = \delta_{i,j} p_{\la,\mu}^{(i)}$,
In view of the definition (\ref{newp}), this means that $\cC^{i,j} = 0$
unless $i=j\equiv \ell(\la,\mu)\pmod{2}$.
In particular, this shows that all the differentials in the
cochain complex are zero, so we actually have that
$\cC^{i,j} = \ext^{i}_{K_\La}(V(\la),L(\mu))_{-j}$.
The corollary follows.
\end{proof}

\phantomsubsection{Koszulity in the finite dimensional case}
Recall from \cite[Definition 1.1.2]{BGS} that
a positively graded associative unital algebra
$K = \bigoplus_{n \geq 0} K_n$ is {\em Koszul} if
\begin{itemize}
\item $K_0$ is a semisimple algebra;
\item
the module $K / K_{> 0}$ has a linear projective resolution,
i.e.
there is an exact sequence
$\cdots \rightarrow \cP_2 \rightarrow \cP_1 \rightarrow \cP_0 \rightarrow K / K_{> 0} \rightarrow 0$
in the category of graded $K$-modules
such that each $\cP_n$ is projective and generated in degree $n$.
\end{itemize}

\begin{Theorem} \label{koszul}
If $|\La| < \infty$ then $K_\La$ is Koszul.
\end{Theorem}

\begin{proof}
Under the assumption that $|\La| < \infty$,
$K_\La$ is a graded
quasi-hereditary algebra
by \cite[Corollary 5.4]{BS}.
Moreover by \cite[Theorem 5.3]{BS} its left standard modules in
the usual sense of quasi-hereditary algebras
are the cell modules $V(\la)$.
Hence Theorem~\ref{pres} establishes that its left standard modules
have linear projective resolutions. Twisting with the
anti-automorphism $*$ from \cite[(4.14)]{BS} we also get that its right standard modules
have linear projective resolutions.
Therefore $K_\La$ is Koszul by \cite[Theorem 1]{ADL}.
(Alternatively, this can be deduced from
\cite[Theorem (3.9)]{CPS3} and Corollary~\ref{indep}.)
\end{proof}

\begin{Corollary}\label{quadratic}
If $|\La| < \infty$ then $K_\La$ is a quadratic algebra.
\end{Corollary}

\begin{proof}
This is \cite[Corollary 2.3.3]{BGS}.
\end{proof}

For the next corollary, we recall the definition of the
{Koszul complex} following \cite[2.6]{BGS}.
Still assuming $|\La| < \infty$, let $k_\La$ denote the semisimple
algebra $(K_\La)_0$
and write $\otimes$ for $\otimes_{k_\La}$.
By Corollary~\ref{quadratic}, we can identify
$K_\La$ with $T(V_\La) / (R_\La)$ where
$T(V_\La)$ is the tensor algebra of
the $k_\La$-module
$V_\La := (K_\La)_1$ and $R_\La$ is the subspace of $V_\La \otimes V_\La$
consisting of the elements whose canonical image in $(K_\La)_2$ is zero.
Let $\cP_0 := K_\La$, $\cP_1 := K_\La \otimes V_\La$ and
\begin{equation}\label{complex}
\cP_n :=
K_\La \otimes
\bigcap_{i=0}^{n-2}
\left(
V_\La^{\otimes i} \otimes R_\La \otimes V_\La^{\otimes (n-2-i)}\right)
\end{equation}
for $n \geq 2$, all
viewed as $K_\La$-modules via the left regular action on the first tensor factor.
The {\em Koszul complex} for $K_\La$ is
\begin{equation}\label{complex2}
\cdots\stackrel{d_2}{\longrightarrow}
\cP_2 \stackrel{d_1}{\longrightarrow} \cP_1 \stackrel{d_0}{\longrightarrow} \cP_0 \stackrel{\eps}{\longrightarrow} K_\La / (K_\La)_{> 0}
\longrightarrow 0
\end{equation}
where
the differential
$d_n$
is the map
$a \otimes v_0 \otimes \cdots \otimes v_n \mapsto
a v_0 \otimes\cdots\otimes v_n$,
and $\eps$ is the natural quotient map.
The following corollary implies that the Koszul complex
is a canonical linear projective resolution
of $K_\La / (K_{\La})_{> 0}$.

\begin{Corollary}\label{kex}
If $|\La| < \infty$ then the Koszul complex for $K_\La$ is exact.
\end{Corollary}

\begin{proof}
This is \cite[Theorem 2.6.1]{BGS}.
\end{proof}

To formulate the final corollary, recall the matrix $P_\La(q)$ from (\ref{pmat}) and
the Cartan matrix $C_\La(q)$ from \cite [(5.6)]{BS}.
Continuing to assume that $|\La| < \infty$,
define the {\em Poincar\'e matrix}
\begin{equation}\label{poine}
E_\La(q) := (e_{\la,\mu}(q))_{\la,\mu \in \La}
\end{equation}
where
$e_{\la,\mu}(q) := \sum_{i \geq 0}
q^i \dim \ext_{K_\La}^i(L(\la),L(\mu))$.
Because $K_\La$
has finite global dimension in our current situation
each $e_{\la,\mu}(q)$ is a polynomial.

\begin{Corollary}\label{poin}
If $|\La| < \infty$ then
$E_\La(q)  = C_\La(-q)^{-1}= P_\La(q)^T P_\La(q)$.
\end{Corollary}

\begin{proof}
The first equality is \cite[Theorem 2.11.1]{BGS}.
The second equality follows
by inverting the
formula $C_\La(-q) = D_\La(-q) D_\La(-q)^T$ from \cite[(5.17)]{BS}
and using Corollary~\ref{invmat}.
\end{proof}

\phantomsubsection{Koszulity in the general case}
Recall a {\em locally unital algebra} $K$ is an
associative algebra with a system $\{e_i\:|\:i \in I\}$
of mutually orthogonal idempotents such that
$K = \bigoplus_{i,j\in I} e_i K e_j$.
There is a natural notion of a
{\em locally unital Koszul algebra}:
a locally unital positively
graded algebra $K = \bigoplus_{n \geq 0} K_n$ such that
\begin{itemize}
\item $K_0$ is a (possibly infinite) direct sum of matrix algebras;
\item the module $K  / K_{>0}$ has a linear projective resolution.
\end{itemize}
Assume for the remainder of the subsection that $\La$ is
block with $|\La| = \infty$.
The algebra $K_\La$ is not unital
but it is locally unital; see \cite[(4.13)]{BS}.
We wish to prove that it is Koszul in the new sense.
(We remark that a detailed treatment of quadratic and Koszul duality
for non-unital algebras, based on the approach of \cite{MVS},
can be found for example in \cite{MOS}.)

Recall the notation $\prec$ from \cite[$\S$4]{BS}.
If $\Ga \prec \La$ then
$\Ga$ is a finite block that is canonically identified with a finite subset of $\La$; see \cite[(4.9)]{BS}.
Moreover
the algebra  $K_\Ga$ is canonically identified with a subalgebra of $K_\La$; see \cite[(4.12)]{BS}.
Any finite subset of $\La$ is contained in some $\Ga \prec \La$.
It follows that any finite dimensional subspace of $K_\La$ is contained
in $K_\Ga$ for some $\Ga \prec \La$.
Since each $K_\Ga$ is quadratic by Corollary~\ref{quadratic}
we deduce:

\begin{Lemma}\label{q2}
$K_\La$ is quadratic.
\end{Lemma}

Let $k_\La := (K_\La)_{0}$, which is an infinite direct sum of copies of the
ground field by \cite[(5.2)]{BS}.
Let $V_\La := (K_\La)_{1}$
and $R_\La \subseteq V_\La \otimes V_\La$ denote the elements
whose canonical image in $(K_\La)_2$ is zero.
Lemma~\ref{q2}
means we can identify $K_\La = T(V_\La) / (R_\La)$.
Now construct the Koszul complex
\begin{equation}\label{newk}
\cdots\stackrel{d_2}{\longrightarrow}
\cP_2 \stackrel{d_1}{\longrightarrow} \cP_1 \stackrel{d_0}{\longrightarrow} \cP_0 \stackrel{\eps}{\longrightarrow} K_\La / (K_\La)_{> 0}
\longrightarrow 0
\end{equation}
in the infinite dimensional
case in exactly the same way as (\ref{complex2}).
Note for $\Ga \prec \La$ that $V_\Ga$ is a subspace
of $V_\La$ and $R_\Ga$ is a subspace of $R_\La$.

\begin{Lemma}\label{tr}
Given $n \geq 0$ and $\la,\mu \in \La$,
there exists $\Ga \prec \La$ such that
$\la,\mu \in \Ga$ and
$e_\mu V_\La^{\otimes n} e_\la = e_\mu V_\Ga^{\otimes n} e_\la$.
Hence
$$
e_\mu\bigcap_{i=0}^{n-2}
\left(V_\La^{\otimes i} \otimes R_\La \otimes V_\La^{\otimes (n-2-i)}\right)
e_\la =
e_\mu
\bigcap_{i=0}^{n-2}
\left(V_\Ga^{\otimes i} \otimes R_\Ga \otimes V_\Ga^{\otimes (n-2-i)}\right)e_\la.
$$
\end{Lemma}

\begin{proof}
Note $e_\mu V_\La^{\otimes n} e_\la$ is spanned by
vectors of the form
$$
(\underline{\la}_0 \nu_1 \overline{\la}_1)
\otimes(\underline{\la}_1 \nu_2 \overline{\la}_2)\otimes
\cdots
\otimes(\underline{\la}_{n-1} \nu_n \overline{\la}_n)
$$
for weights $\la_i, \nu_j \in \La$ such that
$\la_0 = \mu, \la_n = \la$ and
each $\underline{\la}_{i-1} \nu_i \overline{\la}_i$
is an oriented circle diagram of degree one.
By \cite[Lemma 2.4]{BS} there are only finitely many such weights.
Hence we can find $\Ga \prec \La$ such that all possible
$\la_i,\nu_j \in \La$ actually belong to $\Ga$.
Then it is clear that $e_\mu V_\La^{\otimes n} e_\la
= e_\mu V_\Ga^{\otimes n} e_\la$.
The second statement follows by similar considerations.
\end{proof}

\begin{Theorem}
The Koszul complex is exact.
\end{Theorem}

\begin{proof}
The image of $d_0$ is the kernel of $\eps$
and of course $\eps$ is surjective.
So it suffices to show
that the complex of vector spaces
$$
0 \longrightarrow (e_\mu \cP_m e_\la)_m \stackrel{d_{m-1}}{\longrightarrow}
\cdots
\stackrel{d_1}{\longrightarrow}
 (e_\mu \cP_1 e_\la)_m
\stackrel{d_0}{\longrightarrow}
 (e_\mu \cP_0 e_\la)_m
$$
is exact for each $\la,\mu \in \La$ and $m \geq 0$.
The elements of all of the vector spaces appearing in this complex are
linear combinations of
vectors of the form
$$
(\underline{\mu} \nu_0 \overline{\la}_0)
\otimes
(\underline{\la}_0 \nu_1 \overline{\la}_1)\otimes\cdots\otimes
(\underline{\la}_{n-1} \nu_n \overline{\la}_n)
$$
for $0 \leq n \leq m$ and $\la_i,\nu_j \in \La$ with
$\la_n = \la$ such that
$\underline{\mu} \nu_0 \overline{\la}_0$ is an oriented circle
diagram of degree $(m-n)$ and each
$\underline{\la}_{i-1} \nu_i \overline{\la}_i$ is an oriented
circle diagram of degree one.
By \cite[Lemma 2.4]{BS} there are only finitely many such weights.
Hence we can find $\Ga \prec \La$ such that all possible
$\la_i, \nu_j$ belong to $\Ga$.
Just like in the proof of Lemma~\ref{tr}, we deduce that
$(e_\mu \cP_n e_\la)_m
= (e_\mu \cQ_n e_\la)_m$
where $\cQ_n$ denotes the $n$th term in the Koszul complex for
$K_\Ga$.
Now we are done because the Koszul complex for $K_\Ga$
is already known to be exact by Corollary~\ref{kex}.
\end{proof}

\begin{Corollary} $K_\La$ is Koszul.
\end{Corollary}

\begin{proof}
The Koszul complex is a linear projective resolution of
$K_\La / (K_\La)_{> 0}$.
\end{proof}

\begin{Corollary}\label{Es}
Given $\la,\mu \in \La$ and $n \geq 0$,
there exists $\Ga \prec \La$ such that
$\la,\mu \in \Ga$ and
$\ext^n_{K_\La}(L(\la),L(\mu))
\cong
\ext^n_{K_\Ga}(L(\la),L(\mu))$.
\end{Corollary}

\begin{proof}
Multiplying the Koszul complex (\ref{newk}) on the
right by $e_\la$ gives a projective resolution of $L(\la)$.
Then apply the functor $\hom_{K_\La}(?, L(\mu))$ and use
adjointness of tensor and hom to deduce that
$$
\ext^n_{K_\La}(L(\la),L(\mu))
\cong
e_\mu\bigcap_{i=0}^{n-2}
\left(V_\La^{\otimes i} \otimes R_\La \otimes V_\La^{\otimes (n-2-i)}\right)
e_\la,
$$
interpreting the right hand side as
$e_\mu V_\La e_\la$ in case $n=1$ and as $e_\mu k_\La e_\la$
in case $n=0$.
Now apply Lemma~\ref{tr}.
\end{proof}

The previous corollary
implies in particular that each $\ext^n_{K_\La}(L(\la),L(\mu))$
is finite dimensional.
So it still makes sense to define the {Poincar\'e matrix}
$E_\La(q)$
exactly as in (\ref{poine}), though its entries
may now be power series rather than polynomials.

\begin{Corollary}\label{poin2}
$E_\La(q)  = C_\La(-q)^{-1}= P_\La(q)^T P_\La(q)$.
\end{Corollary}

\begin{proof}
In view of the factorisation
\cite[(5.17)]{BS} and Corollary~\ref{invmat},
we just need to show that
$$
\dim
\ext^n_{K_\La}(L(\la),L(\mu))
= \sum_{m=0}^n \sum_{\nu \in \La}
p_{\nu,\la}^{(m)} p_{\nu,\mu}^{(n-m)}
$$
for each fixed $n \geq 0$.
By Corollary~\ref{Es} we can pick $\Ga \prec \La$ (depending on $n$)
such that
$\ext^n_{K_\La}(L(\la),L(\mu))
\cong\ext^n_{K_\Ga}(L(\la),L(\mu))$.
Applying Corollary~\ref{poin} to $\Ga$, it remains to show that
\begin{equation}\label{rems}
\sum_{m=0}^n \sum_{\nu \in \La}
p_{\nu,\la}^{(m)} p_{\nu,\mu}^{(n-m)}
=
\sum_{m=0}^n \sum_{\nu \in \Ga}
p_{\nu,\la}^{(m)} p_{\nu,\mu}^{(n-m)}.
\end{equation}
By the proof of Corollary~\ref{Es},
$\ext^n_{K_\Upsilon}(L(\la),L(\mu))
\cong\ext^n_{K_\Ga}(L(\la,L(\mu))$
for any $\Ga \prec \Upsilon \prec \La$.
Applying Corollary~\ref{poin} to $\Upsilon$,
we deduce that (\ref{rems}) is definitely true if $\La$ is
replaced by any such a block $\Upsilon$.
Hence the equality must also be true for $\La$ itself.
\end{proof}

\section{The double centraliser property}\label{sbijective}

In this section, we are at last able to explain the
full connection between the generalised Khovanov algebra
$H_\La$ and the algebra $K_\La$.
If $\La$ is of infinite defect
then $H_\La = \{0\}$, in which case there is of course no
connection at all.
In all other cases, we are going to prove a
double centralizer property which implies that
$K_\La$ is a quasi-hereditary cover of $H_\La$
in the sense of \cite[Definition 4.34]{Rou}.

\phantomsubsection{Prinjective modules}
To start with, let $\La$ be any block.
We call a $K_\La$-module a
{\em prinjective module} if it is both projective and injective
in the category of
locally finite dimensional graded $K_\La$-modules.
The next theorem is
rather analogus to a theorem of Irving \cite{I1}
in the context of parabolic category $\mathcal O$.
It shows that the prinjective indecomposable modules
are parametrised by
exactly
the set $\La^\circ$ of weights of maximal defect
that appears in the definition (\ref{hla}) of $H_\La$ (no coincidence!).

\begin{Theorem}\label{prin}
For $\la \in \La$, the following conditions are equivalent:
\begin{itemize}
\item[(i)] $\la$ is of maximal defect, i.e. $\la$ belongs to $\La^\circ$;
\item[(ii)] $P(\la)^\circledast \cong P(\la)\langle -2\defect(\la)\rangle$;
\item[(iii)] $P(\la)$ is a prinjective module;
\item[(iv)] $L(\la)$ is isomorphic to a submodule of
$V(\mu)\langle j \rangle$ for some $\mu \in \La$, $j \in \Z$.
\end{itemize}
\end{Theorem}

\begin{proof}
(i)$\Rightarrow$(ii).
Assume that $\la \in \La^\circ$.
There exists a block $\Ga = \{\ga\}$ of defect $0$
and a proper $\La \Ga$-matching $t$ such that
$\overline{\la}$ is the upper reduction of
$\underline{\ga}t$.
This matching $t$ has $\caps(t) = \defect(\la)$ and $\cups(t) = 0$.
By Theorem~\ref{pf},
$$
P(\la) \cong
(G_{\La\Ga}^t P(\ga))\langle\defect(\la)\rangle.
$$
As $P(\ga)$ is irreducible, we have that
$P(\ga)^\circledast \cong P(\ga)$.
Hence by Theorem~\ref{duality}, we deduce that
$P(\la)^\circledast \cong
P(\la) \langle -2 \defect(\la) \rangle$.

(ii)$\Rightarrow$(iii).
This follows by duality
because $P(\la)$ is projective in the category
of locally finite dimensional graded $K_\La$-modules.

(iii)$\Rightarrow$(iv).
As $L(\la)$ is the irreducible head of $P(\la)$,
we deduce that $L(\la)$ is the irreducible socle of
$P(\la)^\circledast$.
Assuming (iii), $P(\la)^\circledast$ is a projective module.
By
\cite[Theorem~5.1]{BS}, projective modules have a filtration by cell modules.
Hence $P(\la)^{\circledast}$ has a submodule isomorphic to
$V(\mu)\langle j \rangle$ for some
$\mu \in \La$ and $j \in \Z$,
such that $L(\la)$ is a submodule of this
$V(\mu)\langle j \rangle$.

(iv)$\Rightarrow$(i).
Take $\mu \in \La$ and
suppose that
$w$ is a vector
spanning an irreducible $K_\La$-submodule
of $V(\mu)$ isomorphic to $L(\la)\langle -j \rangle$ for some
$\la \in \La$ and $j \in \Z$.
Since $e_\la$ acts as $1$ on $L(\la)$
and as zero on all the basis vectors
$(\underline{\nu} \mu|$  of $V(\mu)$ for $\la \neq \nu$,
we may assume actually that
$\underline{\la}\mu$ is an oriented cup diagram and
$w = (\underline{\la} \mu|$.
We need to show that $\la$ is of maximal defect.

Assume for a contradiction that $\la$ is not of maximal defect.
Then we can find a pair of rays
in the diagram $\underline{\la} \mu$ that are oppositely oriented
and neigbouring (in the sense that there are only vertices labelled
$\circ$ or $\times$ at the ends of cups in between).
Define $\nu$ so that $\underline{\nu}$
is obtained from $\underline{\la}$ by replacing these two rays
by a single cup.
We then have that $\underline{\nu}\la$ and $\underline{\nu}\mu$
are oriented cup diagrams,  and
$\nu \neq \la$.
By \cite[Theorem~4.4(iii)]{BS} and (\ref{Actby}),
we see that
$(\underline{\nu} \la \overline{\la}) (\underline{\la} \mu|
=
(\underline{\nu} \mu|$.
This contradicts the assumption that $w = (\underline{\la}\mu|$ spans
a $K_\La$-submodule
of $V(\mu)$.
\end{proof}

\phantomsubsection{``Struktursatz''}
From now on we assume that $\defect(\La) < \infty$.
The following theorem is key
to proving the double centraliser property.

\begin{Theorem}
For each $\la \in \La$, there is an exact sequence
$$
0 \rightarrow P(\la) \rightarrow \cP^0 \rightarrow \cP^1
$$
where $\cP^0$ and $\cP^1$ are
finite direct sums of prinjective indecomposable modules.
\end{Theorem}

\begin{proof}
Assume to start with that $\la$ is maximal in $\La$ with respect to the
Bruhat order.
Then $P(\la) = V(\la)$ by \cite[Theorem~5.1]{BS}.
Let $\mu$ be the weight obtained from $\la$
by swapping $\defect(\La)$ $\up$'s with $\defect(\La)$ $\down$'s.
This ensures that $\mu \in \La^\circ$ and $\la\overline{\mu}$
is an oriented cap diagram of degree $\defect(\La)$.
By \cite[Theorem~5.1]{BS} again we deduce that
$V(\la)$ embeds into $\cP^0 := P(\mu) \langle -\defect(\La)\rangle$
and the quotient $\cP^0 / V(\la)$ has a filtration by cell modules.
By Theorem~\ref{prin} $\cP^0$ is
a prinjective indecomposable module, the socle of $\cP^0 / V(\la)$
is a finite direct sum
of $L(\nu)\langle j \rangle$'s for $\nu \in \La^\circ$,
and
the injective hull of each such $L(\nu)\langle j \rangle$
is the prinjective indecomposable module
$P(\nu)\langle j  - 2 \defect(\nu) \rangle$.
Hence the injective hull $\cP^1$ of $\cP^0 / V(\la)$
is a finite direct sum of prinjective indecomposable modules.
This constructs the desired exact sequence
$0 \rightarrow P(\la) \rightarrow \cP^0 \rightarrow \cP^1$.

Now for the general case, let $\ga$ be the weight obtained
from $\la$ by removing the vertices at the ends of all
caps in the cap diagram $\overline{\la}$.
Then the diagram $\overline{\ga}$ has no caps, hence $\ga$ is
maximal in in its block
$\Ga$.
Let $t$ be the proper $\La\Ga$-matching such that
$\overline{\la}$ is the upper reduction of $t\overline{\ga}$.
The previous paragraph applied to $\ga$
implies there is an exact sequence
$0 \rightarrow P(\ga)\langle \caps(t) \rangle
\rightarrow \cQ^0 \rightarrow \cQ^1$
with $\cQ^0$ and $\cQ^1$ being finite direct sums
of (possible several copies of) prinjective indecomposable modules.
Now apply the
functor $G_{\La\Ga}^t$, noting
$G^t_{\La\Ga}
(P(\ga) \langle \caps(t) \rangle) \cong P(\la)$ by Theorem~\ref{pf}, to
obtain an
exact sequence
$0 \rightarrow P(\la) \rightarrow \cP^0 \rightarrow \cP^1$
where $\cP^i := G^t_{\La\Ga} \cQ^i$.
Since $F$ sends finitely generated prinjectives to
finitely generated prinjectives by Corollary~\ref{fp}
and Theorem~\ref{duality},
each $\cP^i$ is a finite direct sum of prinjective indecomposable
modules, as required.
\end{proof}

The following corollary is analogous to Soergel's
Struktursatz from \cite{Soergel} for regular blocks
of category $\mathcal O$ (hence the title of this subsection).

\begin{Corollary}\label{fullyfaithful}
The truncation functor $e:\mod{K_\La} \rightarrow \mod{H_\La}$
from (\ref{efunc})
is fully faithful on projectives.
\end{Corollary}

\begin{proof}
This follows by the theorem and \cite[Theorem 2.10]{KSX};
see also \cite[Corollary 1.7]{Stp} for a self-contained argument.
\end{proof}

\begin{Corollary}\label{doublecentraliser}
The truncation functor $e:\mod{K_\La} \rightarrow \mod{H_\La}$
from (\ref{efunc}) defines a graded algebra isomorphism
$$
K_\La
\cong \bigoplus_{\la, \mu \in \La}
\hom_{H_\La}(e P(\la), e P(\mu)),
$$
where the right hand side becomes an algebra
with multiplication defined by the rule $f g := g \circ f$
for $f:e P(\la) \rightarrow e P(\mu)$ and
$g:e P(\mu) \rightarrow e P(\nu)$.
\end{Corollary}

\begin{proof}
By the previous corollary, there is a vector space isomorphism
$$
e_\la K_\La e_\mu \stackrel{\sim}{\rightarrow} \hom_{H_\La}(e P(\la), e P(\mu))
$$
sending $x \in e_\la K_\La e_\mu$
to the map
$e P(\la) \rightarrow e P(\mu)$ obtained by applying
the functor $e$ to the homomorphism
$P(\la) \rightarrow P(\mu)$ defined by right multiplication by $x$.
Now take the direct sum of these maps
over all $\la,\mu \in \La$.
\end{proof}

\begin{Remark}\rm
This double centraliser property is most interesting in the case that
$|\La| < \infty$. Then $e = \sum_{\la \in \La^\circ} e_\la$
is an idempotent in $K_\La$,
the generalised Khovanov algebra
$H_\La$ is equal to $e K_\La e \cong \End_{K_\La}(K_\La e)^{\op}$,
the functor $e$ is the obvious truncation functor defined by
multiplication by $e$,
and the double centralizer property shows simply that
$$
K_\La \cong \End_{H_\La}(e K_\La)^{\op}.
$$
If $|\La| = \infty$ (but still $\defect(\La) < \infty$) it often happens
that $H_\La = K_\La$, in which case the double centraliser property is rather vacuous; see \cite[(6.11)]{BS}.
\end{Remark}

\phantomsubsection{Rigidity of cell modules}
Continue to assume that $\defect(\La) < \infty$ and fix $\la \in \La$.
We know already from
Theorem~\ref{prin} that all constituents of the socle of the cell module $V(\la)$ are parametrized by
weights from the set $\La^\circ$.
We want to make this more precise by showing
that $V(\la)$ actually has
irreducible socle.

Let $\la^\circ$ be the unique weight $\mu \in \La^\circ$ such that $\underline{\mu} \la$
is oriented and $\deg(\underline{\mu} \la)$ is maximal.
To compute $\la^\circ$ in practise, start from the diagram $\la$.
Add clockwise cups to the diagram by repeatedly
connecting $\up\down$ pairs of vertices that are neighbours in the sense that there are no vertices in between not yet
connected to cups.
When no more such pairs are left, add nested anti-clockwise cups connecting as many as possible of the remaining vertices together
in $\down\up$ pairs.
Finally add rays (which will be oriented so either all are $\down$ or all are $\up$) to any vertices
left at the end. Then $\la^\circ$ is the weight whose associated cup diagram $\underline{\la^\circ}$ is the cup diagram
just constructed.
For example, if $\la = \down\down\down\up\up\down\up\up\down\down\up$ then
$\la^\circ = \down\down\down\up\down\up\down\down\up\up\up$:
$$
\begin{picture}(30,55)
\put(-88,45){\line(1,0){206}}
\put(-35,45){\oval(20,20)[b]}
\put(5,45){\oval(20,20)[b]}
\put(65,45){\oval(20,20)[b]}
\put(65,45){\oval(60,50)[b]}
\put(25,45){\oval(180,80)[b]}
\put(-85,45){\line(0,-1){40}}
\put(-87.8,45.2){$\scriptstyle\down$}
\put(-67.8,45.2){$\scriptstyle\down$}
\put(-47.8,45.2){$\scriptstyle\down$}
\put(12.2,45.2){$\scriptstyle\down$}
\put(72.2,45.2){$\scriptstyle\down$}
\put(92.2,45.2){$\scriptstyle\down$}

\put(-27.8,40.5){$\scriptstyle\up$}
\put(-7.8,40.5){$\scriptstyle\up$}
\put(32.2,40.5){$\scriptstyle\up$}
\put(52.2,40.5){$\scriptstyle\up$}
\put(112.2,40.5){$\scriptstyle\up$}

\put(-120,43){$\la$}
\put(-120,20){$\underline{\la^\circ}$}
\end{picture}
$$

\begin{Theorem}\label{simplesocle}
For any $\la \in \La$, we have that $\soc V(\la) \cong L(\la^\circ) \langle \deg(\underline{\la^\circ} \la)\rangle$.
\end{Theorem}

\begin{proof}
By the definition of $\la^\circ$ and \cite[Theorem 5.2]{BS},
 $L(\la^\circ) \langle\deg(\underline{\la^\circ} \la)\rangle$
belongs to $\soc V(\la)$. Therefore it suffices to show that the socle of $V(\la)$ is irreducible.

We begin with two easy situations.
First suppose that $\la$ is minimal in the Bruhat order on $\La$.
Then $V(\la) = L(\la)$ and the conclusion is trivial.
Instead suppose that
$\deg(\underline{\la^\circ} \la) = \defect(\La)$.
Then $V(\la) \langle \defect(\La)\rangle$ is a submodule of $P(\la^\circ)$ thanks to \cite[Theorem 5.1]{BS}.
As $P(\la^\circ)$ has irreducible socle by Theorem~\ref{prin}(ii), we deduce that $V(\la)$ has irreducible socle too.

Next suppose that $\La$ does not have a smallest element in the Bruhat order.
This implies that $|\La| = \infty$ and moreover every weight $\la \in \La$
has the property that $\deg(\underline{\la^\circ} \la) = \defect(\La)$. So we are done thanks to
the second easy situation just discussed.
Hence we may assume that $\La$ has a smallest element in the Bruhat order, and proceed by ascending induction on
the Bruhat ordering. The base case of the induction is the first easy situation just mentioned.
For the induction step, take $\la \in \La$ that is not minimal, and let $\mu$ be the weight obtained from $\la$
by interchanging some $\up\down$ pair of vertices that are neighbours in the usual sense. Denote the indices
of this pair of vertices by $i < j$.
We know by induction that $\soc V(\mu) \cong L(\nu)\langle j\rangle$
for some $\nu \in \La^\circ$ and $j \geq 0$.
Moreover as $\nu \in \La^\circ$ the injective hull of $L(\nu)\langle j \rangle$ is isomorphic to
$P(\nu) \langle j-2\defect(\La)\rangle$ thanks to Theorem~\ref{prin}(ii).
Hence there is an embedding $V(\mu) \hookrightarrow P(\nu)\langle j - 2\defect(\La)\rangle$.

Let $t$ be the $\La\La$-matching with a cap joining vertices $i$ and $j$ on the bottom number line,
a cup joining vertices $i$ and $j$ on the top number line, and straight line segments everywhere else.
Let $G^t_{\La\La}$ be the corresponding projective functor.
By Theorem~\ref{vf}, there is an embedding $V(\la) \hookrightarrow G^t_{\La\La} V(\mu)$.
Applying $G^t_{\La\La}$ to our earlier embedding, we deduce that $V(\la)
\hookrightarrow
G^t_{\La\La} P(\nu) \langle j - 2 \defect(\La)\rangle$.
Finally by Theorem~\ref{pf}, we see that
$G^t_{\La\La} P(\nu) \langle j - 2 \defect(\La)\rangle$ is a direct sum of some number of copies of
a single projective module $P(\gamma)$ (possibly shifted in different degrees) for some $\gamma \in \La^\circ$.
Combined with Theorem~\ref{prin}(ii) again, this shows that $\soc V(\la)$ is also a direct sum of some number of copies
of a single irreducible module $L(\gamma)$ (possibly shifted in different degrees).
But $V(\la)$ itself has a multiplicity-free composition series
according to \cite[Theorem 5.2]{BS}. Hence its socle must actually be irreducible.
\end{proof}

\begin{Corollary}
The radical and socle filtrations of cell modules both coincide with the grading filtration.
In particular, cell modules are rigid.
\end{Corollary}

\begin{proof}
This follows from the lemma combined with Lemma~\ref{q2} and \cite[Proposition 2.4.1]{BGS}.
\end{proof}

\section{Kostant modules}

In the last section, we construct an analogue of the
(generalised) BGG resolution from \cite{BGG, Lep} in the diagrammatic setting.
Throughout the section, we work with a fixed block $\La$.

\phantomsubsection{Homomorphisms between neighbouring cell modules}
Recall we say two vertices $i < j$ of a weight $\la \in \La$
are {\em neighbours} if they are separated only by
$\circ$'s and $\times$'s.

\begin{Lemma}\label{homs}
Let $\mu \in \La$ be a weight and $i < j$ be
neighbouring vertices of
$\mu$ labelled $\up$ and $\down$, respectively.
Let $\la$ be the weight obtained from $\mu$
by interchanging the labels on these two vertices.
Then the linear map
\begin{align*}
f_{\la,\mu}:V(\la)\langle 1 \rangle &\rightarrow V(\mu),\\
(c \la| &\mapsto \left\{\begin{array}{ll}
(c \mu|&\text{if $c$ has a cup joining the $i$th and $j$th vertices,}\\
0&\text{otherwise,}
\end{array}\right.
\end{align*}
is a graded
$K_\La$-module homomorphism of degree zero.
Every $K_\La$-module homomorphism
from $V(\la)$ to $V(\mu)$ is a multiple of $f_{\la,\mu}$.
\end{Lemma}

\begin{proof}
The second statement follows at once from the first, since
$f_{\la,\mu}$ is clearly a
non-zero map and the space
$\hom_{K_\La}(V(\la),V(\mu))$ is at most one-dimensional by \cite[(5.12), (5.14)]{BS}.

For the first statement,
recall that the projective indecomposable module
$P := P(\la) = K_\La e_\la$ has a filtration by cell modules
constructed explicitly in \cite[Theorem 5.1]{BS}.
In particular, if we let $M$ (resp.\ $N$) be the submodule of $P$
spanned by all basis vectors
$(c \nu \overline{\la})$ such that $\nu \neq \la$
(resp.\ $\nu \neq \la, \mu$),
then $P / M \cong V(\la)$
and $M / N \cong V(\mu)\langle 1 \rangle$.
Right multiplication by
$(\underline{\la} \mu \overline{\la}) \in K_\La$
defines an endomorphism $f:P\rightarrow P$ that is homogeneous
of degree two.
By \cite[Theorem 4.4(i)]{BS},
$f$ maps
any basis vector $(c \nu \overline{\la})$ of $P$ to
a linear combination of basis vectors
of the form $(c \ga \overline{\la})$ where
$\nu < \ga \geq \mu$. It follows that $f(P) \subseteq M$ and
$f(M) \subseteq N$.
Hence $f$ induces a degree two homomorphism
$$
f_{\la,\mu}:V(\la)
\cong P / M \rightarrow M / N \cong V(\mu)\langle 1 \rangle.
$$
To see that this is
exactly the map $f_{\la,\mu}$ from the statement of the lemma,
it remains to observe by \cite[Theorem 4.4(iii)]{BS} that
$$
f((c \la \overline{\la}))
=
(c \la \overline{\la}) (\underline{\la} \mu \overline{\la})
\equiv \left\{
\begin{array}{ll}
(c \mu \overline{\la}) \pmod{N}&\text{if $c \mu$ is oriented,}\\
0 \hspace{7.5mm}\pmod{N}&\text{otherwise.}
\end{array}\right.
$$
Since $c \la$ is oriented {\em a priori}
and $\la$ differs from $\mu$ only at the vertices $i$ and $j$,
the condition that
$c \mu$ is oriented is equivalent to saying that $c$ has a cup joining the
$i$th and $j$th vertices.
\end{proof}

\phantomsubsection{Kostant weights}
A {\em Kostant weight} means a weight $\mu \in \La$ such that
\begin{equation}
\sum_{i \geq 0} \dim \ext^i_{K_\La}(V(\la),L(\mu)) \leq 1
\end{equation}
for every $\la \in \La$.
The following lemma gives a combinatorial classification of such
weights: they are the weights that are ``$\up\down\up\:\down$-avoiding.''
In view of Remark~\ref{lls}, this result can also be deduced
from existing literature (at least
for bounded weights); see e.g. \cite{BH} and the references therein.

\begin{Lemma}
\label{kostantwts}
For $\mu \in \La$, the following are equivalent:
\begin{itemize}
\item[(i)] $p_{\la,\mu}(q) = q^{\ell(\la,\mu)}$
for all $\la \leq \mu$;
\item[(ii)] $\mu$ is a Kostant weight;
\item[(iii)] $\mu$ is $\up\down\up\:\down$-avoiding, by which we mean it is impossible to find vertices
$i < j < k < l$
whose labels in $\mu$ are
$\up$, $\down$, $\up$ and $\down$, respectively.
\end{itemize}
\end{Lemma}

\begin{proof}
(i)$\Rightarrow$(ii).
This is clear from Corollary~\ref{indep}.

(ii)$\Rightarrow$(iii).
Assume that (iii) is false.
Choose vertices $i < j < k < l$
labelled $\up, \down, \up$ and $\down$, respectively,
so that $(l-i)$ is minimal and $j$ and $k$ are neighbours.
Let $\la$ be the weight obtained from $\mu$ by interchanging
the labels on the $i$th and $l$th vertices.
Then $\overline{\mu}$ has a small cap
connecting the $j$th and $k$th vertices, and
$\ell_j(\la,\mu) = 1$.
We deduce from (\ref{newp}) that $p_{\la,\mu}(1) \geq 2$,
so that $\mu$ is not a Kostant weight by Corollary~\ref{indep}.

(iii)$\Rightarrow$(i).
Assume that $\mu$ is $\up\down\up\:\down$-avoiding, and we are given $\la \leq \mu$.
Let $i < j$ be neighbouring vertices labelled $\down$ and $\up$
in $\mu$, respectively. So there is a small cap in $\overline{\mu}$
joining the $i$th and $j$th vertices.
By the assumption on $\mu$, it is either the
case that no vertex $\leq i$ in $\mu$ is labelled $\up$,
or that no vertex $> i$ in $\mu$ is labelled $\down$.
In the former case, the fact that $\la \leq \mu$
implies that all of the vertices $\leq i$
are labelled in the same way in $\la$ as in $\mu$,
hence $\ell_i(\la,\mu) = 0$.
In the latter case, all the vertices $> i$ are labelled in the same way
in $\la$ as in $\mu$, and of course $\la$ and $\mu$ are in the same block,
so we can  again deduce that
$\ell_i(\la,\mu) = 0$. We deduce from (\ref{newp}) that $p_{\la,\mu}(q) = q^{\ell(\la,\mu)}$, because the label in every
small cap of $\overline{\mu}$ is zero.
\end{proof}

\phantomsubsection{The BGG resolution}
Let
$\mu \in \La$ be any weight.
Following a construction going back to
work of Bernstein, Gelfand and Gelfand \cite{BGG}
and Lepowsky \cite{Lep}, we are going to define a complex
of graded $K_\La$-modules
\begin{equation}\label{bggres}
\cdots
\longrightarrow
V_2 \stackrel{d_1}{\longrightarrow}
V_1 \stackrel{d_0}{\longrightarrow}
V_0 \stackrel{\eps}{\longrightarrow}
L(\mu) \longrightarrow 0,
\end{equation}
where
\begin{equation}\label{modi}
V_n := \bigoplus_{\substack{\la \leq \mu
\\
\ell(\la,\mu) = n}}
V(\la)\langle n \rangle.
\end{equation}
We will refer to this as the {\em BGG complex}. It will turn out that
it is
exact if and only if $\mu$ is a Kostant weight (as described combinatorially by
Lemma~\ref{kostantwts}).

To define the maps in the BGG complex,
note to start with that
$V_0 = V(\mu)$. So we can simply
define the map $\eps$ to be the canonical
surjection $V(\mu) \twoheadrightarrow L(\mu)$.
To define the other
maps we need a little more preparation.
Introduce the notation $\la \rightarrow \nu$
to indicate that $\la \leq \nu$ and $\ell(\la,\nu) = 1$.
In that case,
there is
a canonical homomorphism $f_{\la,\nu}:V(\la)\langle 1 \rangle
\rightarrow V(\nu)$
thanks to Lemma~\ref{homs}.
Call a quadruple of weights $(\la, \nu, \nu', \la')$
a {\em square} if
$$
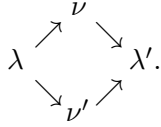
\begin{picture}(-30,48)
\put(-35.5,18.7){$\la$}
\put(10.5,18.7){$\la'.$}
\put(-25.5,29){$\nearrow$}
\put(-25.5,7){$\searrow$}
\put(-2.5,29){$\searrow$}
\put(-2.5,7){$\nearrow$}
\put(-11.5,39.3){$\nu$}
\put(-13.4,-1){$\nu'$}
\end{picture}
$$
By an easy variation on
\cite[Lemma 10.4]{BGG}, it is possible to pick
a sign $\sigma(\la,\nu)$ for each arrow $\la \rightarrow \nu$
such that for every square
the product of the signs associated to its four arrows
is equal to $-1$.
We can now define the differential
$d_n:V_{n+1} \rightarrow V_n$
to be the sum of the maps
$$
\sigma(\la,\nu) f_{\la,\nu}:V(\la)\langle n+1\rangle
\rightarrow V(\nu)\langle n \rangle
$$
for each $\la \rightarrow \nu \leq \mu$ with
$\ell(\nu,\mu) = n$.

Now we can formulate the main result of the section.
It shows that (\ref{bggres})
is indeed a complex for any $\mu \in \La$, and
it is exact
if and only if $\mu$ is a Kostant weight.
Hence, for Kostant weights,
the BGG complex is actually a
{\em resolution} of $L(\mu)$ by multiplicity-free direct sums of cell modules.
Our proof mimics the
general idea of the argument in \cite[Lemma 4.4]{Lep} (see also \cite[Lemma 5]{HK}
and \cite[Proposition 6]{HK}) in the diagrammatic setting.

\begin{Theorem}
The following hold for any $\mu \in \La$:
\begin{itemize}
\item[(i)] $\im\, d_{n+1} = \ker d_n$
for each $n \geq 0$;
\item[(ii)] $\im\, d_0 \subseteq \ker \eps$
with equality if and only if
$\mu$ is a Kostant weight.
\end{itemize}
\end{Theorem}

\begin{proof}
We proceed in several steps.

\vspace{1mm}
\noindent
{\em Step one: $\im\, d_0 \subseteq \ker \eps$.}
This is clear as the image of $d_0$ is contained in the
unique maximal submodule of $V_0$, which is $\ker \eps$.

\vspace{1mm}
\noindent
{\em Step two: $\im\, d_{n+1} \subseteq \ker d_n$ for any $n \geq 0$.}
Given $\la \leq \mu$,
it is convenient to let
$$
i_\la:V(\la)\langle m \rangle \hookrightarrow
V_m,
\qquad
p_\la:V_m \twoheadrightarrow V(\la)\langle m \rangle
$$
be the natural inclusion and projection maps,
where $m := \ell(\la,\mu)$.
We need to show $d_n \circ d_{n+1} = 0$,
which follows if we check that the map
$$
r_{\la,\la'} := p_{\la'} \circ d_n \circ d_{n+1} \circ i_\la:
V(\la)\langle n+2 \rangle \rightarrow V(\la') \langle n \rangle
$$
is zero for every $\la, \la' \leq \mu$ with
$\ell(\la,\mu) = n+2$ and $\ell(\la',\mu) = n$.
The map $r_{\la,\la'}$ is trivially zero
unless there exists a weight
$\nu$ such that $\la \rightarrow \nu \rightarrow \la'$.
Suppose that $\nu$ is such a weight, and
let $i < j$ (resp.\ $k < l$) be the pair of neighbouring
vertices that are labelled
differently in $\nu$ compared to $\la$ (resp.\ in $\la'$ compared to $\nu$).
If $\{i,j\} \cap \{k,l\} \neq \varnothing$, then
$\nu$ is the unique weight such that $\la \rightarrow \nu
\rightarrow \la'$, and it is clear from the explicit description
from Lemma~\ref{homs} that
$$
r_{\la,\la'} = \sigma(\la,\nu) \sigma(\nu,\la')
f_{\nu,\la'} \circ f_{\la,\nu}
$$
is zero.
If $\{i,j\} \cap \{k,l\}  = \varnothing$, then there is just
one more weight $\nu' \neq \nu$ with $\la \rightarrow \nu' \rightarrow \la'$,
namely, the weight obtained from $\la$ be interchanging the labels on
the $k$th and $l$th vertices. In this case,
$(\la,\nu,\nu',\la')$ is a square, and
$$
r_{\la,\la'} = \sigma(\la,\nu) \sigma(\nu,\la') f_{\nu,\la'} \circ f_{\la,\nu}
+ \sigma(\la,\nu') \sigma(\nu',\la')f_{\nu',\la'} \circ f_{\la,\nu'}.
$$
This time, Lemma~\ref{homs} gives that
$f_{\nu,\la'} \circ f_{\la,\nu} = f_{\nu',\la'} \circ f_{\la,\nu'}$.
Hence $r_{\la,\la'} = 0$ as
$\sigma(\la,\nu) \sigma(\nu,\la')
=
-\sigma(\la,\nu') \sigma(\nu',\la')$ by the choice of signs.

\vspace{1mm}
\noindent
{\em Step three: $\im\, d_{n+1} = \ker d_n$ for any $n \geq 0$.}
It suffices to show that
$e_\gamma$ annihilates $\ker d_n / \im\, d_{n+1}$
for all $\gamma \in \La$.
If $\gamma \not< \mu$ then $e_\gamma V(\la)  = \{0\}$
for all $\la < \mu$, and the conclusion is trivial.
So we may assume that $\gamma < \mu$. This means that
$\gamma$ is not maximal in the Bruhat order, so we can
pick neighbouring vertices $i < j$ of $\gamma$ that are labelled
$\down$ and $\up$, respectively.
For any $m \geq 0$ and symbols $a,b \in \{\down, \up\}$, let
\begin{align*}
\La_m^{a b} &:= \{\la \leq \mu\:|\:\ell(\la,\mu) = m
\text{, vertices $i, j$ of $\la$
are labelled $a, b$, respectively}\}.
\end{align*}
Note that there is a bijection
$\La^{\down\up}_{m+1} \rightarrow \La^{\up\down}_m,
\la \mapsto \la^s$,
where $\la^s$ denotes the weight
obtained from $\la$ by interchanging the labels on the
$i$th and $j$th vertices. Also for any $\la \leq \mu$
with $\ell(\la,\mu) = m$,
the space $e_\gamma V(\la) \langle m \rangle$ is one-dimensional
with basis $(\underline{\gamma} \la|$ if $\la \in \La^{\down\up}_m \cup
\La^{\up\down}_m$,
and $e_\gamma V(\la)\langle m \rangle = \{0\}$ if $\la \in \La^{\up\up}_m \cup \La^{\down\down}_m$.
Finally, for $\la \in \La^{\down\up}_{m+1}$, we have that
\begin{equation}\label{use}
d_m \left((\underline{\gamma} \lambda|\right) \equiv
\sigma(\lambda,\lambda^s) (\underline{\gamma} \lambda^s|
\pmod{\bigoplus_{\nu \in \La^{\down\up}_m} V(\nu)\langle m\rangle}.
\end{equation}

Now take $x \in e_\gamma(\ker d_n)
\subseteq V_{n+1}$. Using the observations just made, we can
write
$x = \sum_{\la \in \La^{\down\up}_{n+2}} a_\la (\underline{\gamma} \la^s|
+ \sum_{\la \in \La^{\down\up}_{n+1}} b_\la (\underline{\gamma} \la|$
for unique $a_\la,b_\la \in \C$.
By (\ref{use}), if we let
$y := \sum_{\la \in \La_{n+2}^{\down\up}} a_\la \sigma(\la,\la^s)
(\underline{\gamma} \la|$,
then
$$
d_{n+1}(y) \equiv x \pmod{\bigoplus_{\nu \in \La^{\down\up}_{n+1}} V(\nu) \langle n+1 \rangle}.
$$
Hence
$x = d_{n+1}(y) + \sum_{\la \in \La^{\down\up}_{n+1}} b'_\la (\underline{\gamma} \lambda|$
for some $b'_\la \in \C$.
Finally we apply the map $d_n$ and use (\ref{use}) again to see that
$$
0=d_n(x) \equiv \sum_{\la \in \La^{\down\up}_{n+1}} b'_\la
\sigma(\la,\la^s) (\underline{\ga}
\la^s| \pmod{\bigoplus_{\nu \in \La^{\down\up}_n} V(\nu) \langle n \rangle}.
$$
It follows that $b'_\la = 0$ for all $\la \in \La^{\down\up}_{n+1}$,
hence $x = d_{n+1} (y)$. This shows that
$e_\gamma (\ker d_n) \subseteq e_\gamma (\im d_n)$, as required.

\vspace{1mm}
\noindent
{\em Step four: $\im\, d_{0} = \ker \epsilon$ if and only if
$\mu$ is a Kostant weight.}
By what we have established so far,
$\im\, d_0 = \ker \epsilon$ if and only if the Euler characteristic
of our complex (computed
in $[\mod{K_\La}]$)
is zero, i.e.
$[L(\mu)] = \sum_{\la \leq \mu} (-q)^{\ell(\la,\mu)}
[V(\la)]$.
By Corollary~\ref{invmat} and \cite[(5.14)]{BS}, we know that
$[L(\mu)] = \sum_{\la \leq \mu} p_{\la,\mu}(-q) [V(\la)]$.
Hence $\im\, d_0 = \ker \epsilon$ if and only if
$p_{\la,\mu}(-q) = (-q)^{\ell(\la,\mu)}$ for each $\la \leq \mu$.
Applying
Lemma~\ref{kostantwts}
this is if and only if $\mu$ is a Kostant weight.
\end{proof}

\end{document}